\documentclass[a4paper,12pt]{amsart}
\usepackage{amsmath, amssymb, amsthm}
\usepackage{verbatim}
\usepackage{hyperref}
\usepackage{graphicx}
\usepackage{xy}

\newcounter{theorem}
\newtheorem{thm}[theorem]{Theorem}
\newtheorem{lemma}[theorem]{Lemma}
\newtheorem{prop}[theorem]{Proposition}
\newtheorem{cor}[theorem]{Corollary}
\newtheorem{defn}[theorem]{Definition}

\newtheorem{question}[theorem]{Question}

\theoremstyle{remark}
\newtheorem*{remark*}{Remark}
\newtheorem{remark}[theorem]{Remark}
\newtheorem{example}[theorem]{Example}

\numberwithin{equation}{section}
\numberwithin{theorem}{section}

\newcommand{\e}{\epsilon}
\newcommand{\dl}{\delta}

\newcommand{\R}{\mathbb{R}}

\newcommand{\N}{\mathbb{N}}
\renewcommand{\setminus}{\backslash}

\newcommand{\id}{\mathrm{id}}

\newcommand{\ccite}[2]{\cite[#1]{#2}}

\def\X{\mathcal{X}}
\def\Y{\mathcal{Y}}
\def\Leps{\mathrm{Commut}}
\newcommand{\diam}{\mathrm{diam}}

\begin{document}

\newcommand{\SV}{\mathrm{VL}}
\newcommand{\dn}{\mathrm{dim}_{\mathrm{nuc}}}
\newcommand{\asdim}{\mathrm{asdim}}

\newcommand{\oomega}{\infty}

\title{Relative commutant pictures of Roe algebras}
\author{J\'an \v{S}pakula}
\address{J\'an \v{S}pakula \\
Mathematical Sciences \\
University of Southampton \\
Highfield \\
Southampton, SO17 1BJ \\
United Kingdom}
\email{jan.spakula@soton.ac.uk}
\author{Aaron Tikuisis}
\address{Aaron Tikuisis \\
Department of Mathematics and Statistics \\
University of Ottawa \\
Ottawa, Ontario, K1N 6N5 \\
Canada}
\email{aaron.tikuisis@uottawa.ca}

\begin{abstract}
  Let $X$ be a proper metric space, which has finite asymptotic dimension in the sense of Gromov (or more generally, straight finite decomposition complexity of Dranishnikov and Zarichnyi).
New descriptions are provided of the Roe algebra of $X$: (i) it consists exactly of operators which essentially commute with diagonal operators coming from Higson functions (that is, functions on $X$ whose oscillation tends to $0$ at $\infty$), and (ii) it consists exactly of quasi-local operators, that is, ones which have finite $\e$-propogation (in the sense of Roe) for every $\e>0$.
These descriptions hold both for the usual Roe algebra and for the uniform Roe algebra.
\end{abstract}

\maketitle

\section{Introduction}
The Roe algebra is a C*-algebra constructed from a proper metric space, which encodes ``coarse'' or ``large-scale'' properties of the space (in the sense of Gromov).
In typical applications, the space may be a complete, non-compact Riemannian manifold with bounded geometry, or a finitely generated group with the word metric.
The origins of this construction come from index theory, reflecting the insight that the Roe algebra is large enough to contain indices of many operators with which one wants to do index theory -- such as geometric differential operators -- yet small enough to have interesting and informative $K$-theory.
It plays a central role in the coarse Baum--Connes conjecture, the study (and partial confirmation) of which has been a fruitful endeavor, leading to significant results concerning the Novikov conjecture and the scalar curvature of Riemannian manifolds \cite{Dranishnikov:ScalarCurvatureConjecture, GuentnerKaminker, KasparovYu,Roe:JDG1, Roe:CBMS, Schick:TopologyOfPSC, SkandalisTuYu, Wright:C0CoarseGeometry, Yu:ZeroInTheSpectrumConjecture, Yu:NCfad, Yu:UnifEmbeddableBCC}.
It furthermore appears in work on the essential spectrum of Hamiltonian operators of quantum systems, Schr\"odinger operators, and various other operators, which are affiliated to the appropriate versions of Roe algebras \cite{Georgescu:JFA,Georgescu:arXiv,GeorgescuIftimovici,RRS}.

Roughly, the Roe algebra consists of bounded, locally compact operators on something like $L^2(X)$ (where $X$ is the underlying space) which can be approximated by those with finite propogation.
Here an operator $a$ has ``finite propogation'' if it is localized near to the diagonal; one way of making this precise is, that there exists $R>0$ such that for any $f,f' \in C_b(X)$ (acting on $L^2(X)$ as diagonal operators -- by multiplication), if the supports of $f$ and $f'$ are separated by a distance of at least $R$, then $faf'=0$.
Operators in the Roe algebra are required to be approximated in the \emph{operator norm} by these finite propogation operators.

The Roe algebra is an interesting object to study from an operator algebraic perspective: that is, asking about its structure as an operator algebra, and how this structure relates to the proper metric space from which it is constructed.
For example, Ozawa showed that exactness of a group can be characterized by amenability of the corresponding uniform Roe algebra (\cite{Ozawa:AmenableActions}).
The question of how much of the large--scale structure of a space is remembered by the Roe algebra, was partially answered by JS and Rufus Willett: given two uniformly discrete proper metric spaces with Yu's property A, their Roe algebras are $^*$-isomorphic if and only if the spaces are coarsely equivalent (\cite{SpakulaWillett:Rigidity}).
In \cite{WinterZacharias:NucDim}, Winter and Zacharias showed an interesting one-way connection between the asymptotic dimension of a metric space and the nuclear dimension of the corresponding uniform Roe algebra; the latter is a numerical invariant for amenable C*-algebras which is crucial in recent results in the classification of amenable C*-algebras.
Their result is that the nuclear dimension of the Roe algebra is at most the asymptotic dimension of the underlying uniformly discrete proper metric space, and they asked the (still open) question of whether the reverse inequality also holds.

In this paper, we look at a fundamental question: exactly which operators \emph{are} in the Roe algebra?
In \cite{Roe:CBMS}, Roe defined the concept of finite $\e$-propogation for an operator $a$ on $L^2(X)$, as the following variant of finite propogation: $a$ has finite $\e$-propogation if there exists $R>0$ such that for any $f,f' \in C_b(X)$, if the supports of $f$ and $f'$ are separated by a distance of at least $R$, then $\|faf'\| \leq \e\|f\|\cdot\|f'\|$.
Operators with finite $\e$-propogation for all $\e>0$ have also been called \emph{quasi-local operators} in the literature (originally from \cite[Page 100]{Roe:JDG1}).
It is a straightforward observation that, although limits of finite propogation operators need not have finite propogation, limits of finite $\e$-propogation operators have finite $\e$-propogation (that is, the set of quasi-local operators is norm-closed).
Therefore, all operators in the Roe algebra are quasi-local.

The question we address is the converse: if an operator is quasi-local, is it in the Roe algebra, i.e., is it approximated by operators with finite propogation?
We provide an affirmative answer in the situation that the space has finite asymptotic dimension (as predicted by Roe), and more generally under the hypothesis of straight finite decomposition complexity of Dranishnikov and Zarichnyi \cite{DranishnikovZarichnyi:HaversC}. The latter is a weaker version of the ``classical'' finite decomposition complexity, as defined by Guentner, Tessera, and Yu \cite{GTY:Inv,GTY:GGD}.

A motivation for asking whether quasi-local operators are in the Roe algebra, pointed out to the authors by Alexander Engel, is that whereas indices of genuine differential operators are known to be in the Roe algebra, corresponding arguments only shows that indices of \emph{pseudo-}differential operators (using their most natural representative) are quasi-local (see \cite[Section 2]{Engel:PseudoDO}).
Since the Roe algebra is better studied and understood than the C*-algebra of all quasi-local operators, it is interesting and useful to know that a pseudo-differential operator belongs to it; indeed, our result answers \cite[Question 6.4]{Engel:PseudoDO} under the assumption of straight finite decomposition complexity (this sort of assumption is anticipated in the question).

Additionally, we expose that the above question can be reformulated in terms of essential commutation with Higson functions, or in terms of relative commutants.

So far we have been a bit vague about what we mean by the Roe algebra (hiding behind the phrase ``something like $L^2(X)$'').
This is because in the literature there are two different versions of the Roe algebra, the ``Roe algebra'' and the ``uniform Roe algebra''.
Our results apply to both of these C*-algebras, and indeed our main theorem is formulated in a way that encompasses both, as well as the ``uniform algebra'' $UC^*(X)$.
The main result was proven by Lange and Rabinovich for the uniform Roe algebra of $\mathbb Z^d$ in \cite{LangeRabinovich}.
Engel proved a special case of the result, namely that for discrete groups $G$ that are lattices in a Riemannian manifold with bounded geometry and polynomial volume growth, quasi-local operators that decay sufficiently quickly are in the Roe algebra (\cite[Corollary 2.33]{Engel:Rough}).\footnote{In fact, Engel proved the result for quasi-local operators that decay sufficiently on \emph{any} Riemannian manifold with bounded geometry and polynomial volume growth. For groups, polynomial growth implies virtual nilpotency (\cite{Gromov:PolyGrowth}), which in turn implies finite asymptotic dimension (\cite[Corollary 68]{BellDranishnikov}). To our knowledge, it is not known whether polynomial volume growth implies finite asymptotic dimension (or even (straight) finite decomposition complexity) outside of the case of groups.}

Let us now summarize the argument behind the main result: that quasi-local operators are in the Roe algebra (assuming straight finite decomposition complexity).
Suppose for simplicity that $X$ is a discrete space with asymptotic dimension at most $1$ -- for example a finitely generated free group. This case is much more restricted than finite decomposition complexity, but still difficult enough to allow us to convey the main ideas.
Let $a$ be a quasi-local operator.
Asymptotic dimension at most $1$ will allow us to decompose the space $X$ into $2$ pieces, $X^{(0)}$ and $X^{(1)}$, each piece being a disjoint union of sets that are far apart from each other and uniformly bounded in diameter.
The characteristic functions $e^{(0)},e^{(1)}$ of these pieces produce a partition of unity, and divides $a$ into a sum of four pieces: $e^{(i)}ae^{(j)}$ over $i,j=0,1$.
Each $e^{(i)}ae^{(i)}$ looks roughly like an infinite block matrix, indexed by the pieces from $X^{(i)}$.
The hypothesis that $a$ is quasi-local (finite $\e$-propogation) gives us a lot of control over the norm of the non-diagonal entries of this matrix, and a conditional expectation argument allows us to conclude that $e^{(i)}ae^{(i)}$ is not far away from its ``restriction'' to the diagonal (provided that the pieces in $X^{(i)}$ are sufficiently well separated), see Corollary \ref{cor:CommuteCEEst}.
Since the pieces of the $X^{(i)}$ are uniformly bounded, the operator we get by expecting onto the diagonal has genuinely finite propogation.
An algebraic trick allows us to view the asymmetric pieces $e^{(i)}ae^{(j)}$ as matrices in a similar way, so that we can likewise approximate each of them by finite propogation operators.
In this way, we approximate $a$ as a sum of four operators with finite propogation.

\textit{Outline.} In Section \ref{sec:Defns} we introduce our general setup, with the Roe algebra, the uniform Roe algebra, and the uniform algebra as examples.
We then state the main result, Theorem \ref{thm:MainThm}, in the language of our general setup.
We give some background on asymptotic dimension and (straight) finite decomposition complexity in Section \ref{sec:CoarseGeom}.
The equivalence between quasi-locality and the relative commutant-type property is fairly straightforward, and laid out in Section \ref{sec:IiffII}.
We use a more technical formulation of quasi-locality as a stepping stone towards proving that it implies being in the Roe algebra (assuming straight finite decomposition complexity), a proof that is carried out in Section \ref{sec:IimpliesIV}.
In Section \ref{sec:Higson}, we prove that the relative commutant-type property is equivalent to essential commutation with Higson functions.
The final section, Section \ref{sec:MoreAboutVL}, is concerned with the commutative (but non-separable) C*-algebra $\SV_\oomega(X)$ that arises in our relative commutant-type property, looking at how well it determines $X$ (up to coarse equivalence), and at its nuclear dimension (roughly, the covering dimension of its spectrum).

\textit{Acknowledgments.}
AT was supported by EPSRC EP/N00874X/1.
JS was supported by Marie Curie FP7-PEOPLE-2013-CIG Coarse Analysis (631945).
We would like to thank Ulrich Bunke, Alexander Engel, John Roe, Thomas Weighill, Stuart White, and Rufus Willett for comments and discussion relating to this piece.

\section{Definitions and the main result}
\label{sec:Defns}
Let $A$ be a C*-algebra.
We denote by $A_1$ the closed unit ball of $A$.
For $a,b \in A$ and $\e>0$, we write $a \approx_\e b$ to mean $\|a-b\|\leq \e$.
Define
\[ A_\infty := l^\infty(\mathbb N,A)/\{(a_n)_{n=1}^\infty \in l^\infty(\mathbb N,A): \lim_{n\to\infty} \|a_n\|=0\}, \]
which is a C*-algebra.

We now set up a general situation to which our main result applies, encompassing both Roe algebras and uniform Roe algebras, as well as uniform algebras (see Example \ref{ex:Setup}).
Subsequently, we will state our main result in its full generality (Theorem \ref{thm:MainThm})

\begin{defn}
Let $X$ be a proper metric space.
By an \emph{$X$-module}, we mean a Hilbert space $\mathcal H$ and an injective unital $^*$-homomorphism $C_b(X) \to \mathcal B(\mathcal H)$, which is strictly continuous when viewing $C_b(X)$ and $\mathcal B(\mathcal H)$ as multiplier algebras of $C_0(X)$ and $\mathcal K(\mathcal H)$ respectively.
We shall suppress the $^*$-homomorphism $C_b(X) \to \mathcal B(\mathcal H)$, and treat $C_b(X)$ as a C*-subalgebra of $\mathcal B(\mathcal H)$.

For $R\geq 0$, an operator $a \in \mathcal B(\mathcal H)$ has \emph{propogation at most $R$} if for any $f,f' \in C_b(X)$, if the supports of $f$ and $f'$ are $R$-disjoint then $faf'=0$.
For $R\geq 0$ and $\e>0$, an operator $a \in \mathcal B(\mathcal H)$ has \emph{$\e$-propogation at most $R$} if for any $f,f' \in C_b(X)_1$, if the supports of $f$ and $f'$ are $R$-disjoint then $\|faf'\|<\e$.
An operator $a \in \mathcal B(\mathcal H)$ is \emph{quasi-local} if for every $\e>0$, it has finite $\e$-propogation.
\end{defn}

\begin{defn}
\label{def:BlockCutdown}
Let $X$ be a proper metric space and let $\mathcal H$ be an $X$-module.
Given an equicontinuous family $(e_j)_{j \in J}$ of positive contractions in $C_b(X)$ with pairwise disjoint supports, define the \emph{block cutdown} map $\theta_{(e_j)_{j\in J}}:\mathcal B(\mathcal H) \to \mathcal B(\mathcal H)$ by
\[ \theta_{(e_j)_{j\in J}}(a) := \sum_{j \in J} e_jae_j \]
(using disjointness of the supports and the fact that the family is contractive, the right-hand sum converges in the strong operator topology).

Let $B \subseteq \mathcal B(\mathcal H)$ be a C*-subalgebra such that $C_b(X)BC_b(X)=B$.
$B$ is \emph{closed under block cutdowns} if $\theta_{(e_j)_{j\in J}}(B) \subseteq B$ for every equicontinuous family $(e_j)_{j \in J}$ of positive contractions from $C_b(X)$ with pairwise disjoint supports.
\end{defn}

For an equicontinuous family $(e_j)_{j \in J}$ of positive contractions from $C_b(X)$ with pairwise disjoint supports, the block cutdown map $\theta_{(e_j)_{j\in J}}$ defined above is evidently completely positive and contractive (c.p.c.).
Note that multiplication by $C_b(X)$ commutes with block cutdowns:
\begin{equation*}
f\theta_{(e_j)_{j\in J}}(a) = \theta_{(e_j)_{j\in J}}(fa) \quad \text{and}\quad 
\theta_{(e_j)_{j\in J}}(a)f = \theta_{(e_j)_{j\in J}}(af)
\end{equation*}
for $f \in C_b(X)$ and $a \in \mathcal B(\mathcal H)$.
Also note that
\begin{equation}
\label{eq:BlockNorm}
\|\theta_{(e_j)_{j\in J}}(a)\| = \sup_{j \in J} \|e_jae_j\|.
\end{equation}

Note that, if $(e_j)_{j \in J}$ is an equicontinuous family of positive contractions from $C_b(X)$ with uniformly bounded, pairwise disjoint supports, then $\theta_{(e_j)_{j \in J}}(a)$ has finite propogation, for every $a \in \mathcal B(\mathcal H)$.

\begin{defn}
Let $X$ be a proper metric space, $\mathcal H$ an $X$-module, and let $B \subseteq \mathcal B(\mathcal H)$ be a C*-subalgebra such that $C_b(X)BC_b(X)=B$, and which is closed under block cutdowns.
Define
\begin{enumerate}
\item $\mathrm{Roe}(X,B):=\overline{\{b \in B: b \text{ has finite propogation}\})}^{\|\cdot\|}$, and
\item $\mathcal K(X,B):= \overline{C_0(X)BC_0(X)}$.
\end{enumerate}
If, in addition, we have
\begin{equation}\label{eq:ideal-cond}
[C_0(X),B]\subseteq \mathcal K(X,B),
\end{equation}
we shall call $\mathrm{Roe}(X,B)$ a \emph{Roe-like algebra} of $X$.
\end{defn}

\begin{remark}
  The condition \eqref{eq:ideal-cond} implies that $\mathcal K(X,B)$ is an ideal in $\mathrm{Roe}(X,B)$. It is automatically satisfied in all the examples below, where in fact $C_0(X)B \subseteq \mathcal K(X,B)$ (and $\mathcal K(X,B)$ turns out to be the ideal of compact operators). Finally, it is not needed for the substantial part of this piece, so we shall explicitly refer to it when needed.
\end{remark}

\begin{example}
\label{ex:Setup}
Let $X$ be a uniformly discrete proper metric space.
Let $\mathcal H'$ be an infinite dimensional, separable Hilbert space.
Set $\mathcal H_u:=l^2(X)$ and $\mathcal H:=l^2(X,\mathcal H')$; $C_b(X)$ acts on both of these by pointwise multiplication, making them $X$-modules.
\medskip

(i)
With $B_u:=\mathcal B(\mathcal H_u)$, we see that $C_b(X)B_uC_b(X) = B_u$, and $B_u$ is closed under block cutdowns.
In this case, $\mathrm{Roe}(X,B_u)=C^*_u(X)$, the uniform Roe algebra, and $\mathcal K(X,B_u)=\mathcal K(\mathcal H_u)$.
Since $C_0(X) \subseteq \mathcal K(\mathcal H_u)$, it follows that $C_0(X)B_u\subseteq \mathcal K(\mathcal H_u) = \mathcal K(X,B_u)$.
\medskip

(ii)
Set $B$ equal to the set of all $b \in \mathcal B(\mathcal H)$ which are locally compact, in the sense that for every $f \in C_0(X)$,
\[ fb, bf \in \mathcal K(\mathcal H). \]
We see that $C_b(X)BC_b(X) = B$, and $B$ is closed under block cutdowns.
Then $\mathrm{Roe}(X,B)=C^*(X)$, the Roe algebra, and $\mathcal K(X,B)=\mathcal K(\mathcal H)$.
\medskip

(iii)
Assume that $X$ has bounded geometry.
Set $B_0$ equal to the closure of the set of all $b =(b_{x,y})_{x,y \in X} \in \mathcal B(\mathcal H)$ for which the rank of $b_{x,y} \in \mathcal B(\mathcal H')$ is uniformly bounded.
When $b=(b_{x,y})_{x,y\in X} \in \mathcal B(\mathcal H)$ has entries with rank bounded by $k$, then so does any block cutdown map applied to $b$.
Since each block cutdown map is continuous, it follows that $B_0$ is closed under block cutdowns.
Continuity of multiplication ensures that 
\[ C_b(X)B_0C_b(X) = B_0. \]
When $X$ has bounded geometry, then $\mathrm{Roe}(X,B_0)=UC^*(X)$, the uniform algebra of $X$, defined as the closure of finite propogation operators $b = (b_{x,y})_{x,y\in X} \in \mathcal B(\mathcal H)$ for which the rank of $b_{x,y}$ is uniformly bounded.

To see this, it is clear that $\mathrm{Roe}(X,B_0)$ contains $UC^*(X)$.
To show $\mathrm{Roe}(X,B_0) \subseteq UC^*(X)$, it suffices to check that every finite propogation operator $a \in B_0$ is contained in $UC^*(X)$.
For such $a$, say its propogation is less than $R >0$.
Set
\[ K:=\sup_{x \in X} |B_R(x)|, \]
which is finite due to the hypothesis of bounded geometry.
Define $E_R:\mathcal B(\mathcal H) \to \mathcal B(\mathcal H)$ by $E_R\left((b_{x,y}\right)_{x,y\in X}) := (c_{x,y})_{x,y \in X}$ where 
\[ c_{x,y} := \begin{cases} b_{x,y}, \quad &d(x,y) < R; \\ 0,\quad &d(x,y) \geq R. \end{cases} \]
Note that $\left\|E_R\left((b_{x,y})_{x,y\in X}\right)\right\| \leq K\left\|(b_{x,y})_{x,y\in X}\right\|$ (this is a straightforward argument, see e.g., the proof of \cite[Lemma 8.1]{WinterZacharias:NucDim}), so that in particular, $E_K$ is continuous.
Also note that $E_K(a)=a$.
Since $a \in B_0$, it is a limit of a sequence of operators $b_n=(b^n_{x,y})_{x,y\in X}$ such that for each $n$, there exists $K_n$ bounding the rank of $b^n_{x,y}$ over all $x,y\in X$.
Thus the same bound $K_n$ applies to $E_K(b_n)$ so that $E_K(b_n) \in UC^*(X)$.
By continuity of $E_K$, $a=\lim_{n \to \infty} E_K(b_n) \in UC^*(X)$.

In this example, we also have $\mathcal K(X,B_0)=\mathcal K(\mathcal H)$, and since $B_0 \subseteq B$ (from (ii)), $C_0(X)B \subseteq \mathcal K(\mathcal H) = \mathcal K(X,B_0)$
\medskip

(iv)
Generalizing (ii), let $X$ be any proper metric space and let $\mathcal H$ be an adequate $X$-module in the sense of \cite[Definition 3.4]{Roe:CBMS}.
Recall that an operator $b \in \mathcal B(\mathcal H)$ is \emph{locally compact} if $C_0(X)b, bC_0(X) \subseteq \mathcal K(\mathcal H)$.
Set $B$ equal to the set of all locally compact, bounded operators.
One can easily see that $C_b(X)BC_b(X)=B$; it is also true that $B$ is closed under block cutdowns.

To see this, let $b \in \mathcal B(\mathcal H)$ be locally compact with $\|b\|\leq1$, let $(e_j)_{j \in J}$ be an equicontinuous family of positive contractions in $C_b(X)$ with pairwise disjoint supports, and set $b':= \theta_{(e_j)_{j \in J}}(b)$, which we must prove is locally compact.
As $\mathcal K(\mathcal H)$ is closed, it suffices to show that for any $f \in C_c(X)$ with $\|f\|\leq1$, $fb', b'f \in \mathcal K(\mathcal H)$.
Given $\e>0$, note that
\[ b' \approx_{2\e} \theta_{((e_j-\e)_+)_{j \in J}}(b), \]
where $(e_j-\e)_+ \in C_b(X)$ is given by $(e_j-\e)_+(x):=\max\{e_j(x)-\e,0\}$.
By equicontinuity and pairwise disjointness of the family $(e_j)$, we may choose $\dl$ such that if $d(x,y)<\dl$ and $j \neq j'$, then at most one of $e_j(x)$ or $e_{j'}(y)$ can be nonzero.
Thus if $f \in C_c(X)$, then by compactness of its support, there are only finitely many $j$ for which $f(e_j-\e)_+\neq 0$.
Consequently,
\begin{align*}
fb' 
&\approx_{2\e} f\theta_{((e_j-\e)_+)_{j\in J}}(b)  \\
&= f\sum_{j \in J}(e_j-\e)_+b(e_j-\e)_+ \\
&= \sum_j f(e_j-\e)_+b(e_j-\e)_+,
\end{align*}
and as this is a finite sum of elements of $\mathcal K(\mathcal H)$, it is itself in $\mathcal K(\mathcal H)$.
As $\mathcal K(\mathcal H)$ is closed and $\e>0$ is arbitrary, it follows that $fb' \in \mathcal K(\mathcal H)$.
Likewise, $b'f \in \mathcal K(\mathcal H)$, establishing that $b'$ is locally compact, and therefore that $B$ is closed under block cutdowns.

In this example, we get $\mathrm{Roe}(X,B)=C^*(X)$, the Roe algebra, and $\mathcal K(X,B)=\mathcal K(\mathcal H)=C_0(X)B$.
\end{example}
\bigskip

\begin{defn}
Let $X$ be a metric space.
A bounded sequence $\left(f_n\right)_{n=1}^\infty$ from $C_b\left(X\right)$ is \emph{very Lipschitz} if, for every $L>0$, there exists $n_0$ such that $f_n$ is $L$-Lipschitz for all $n \geq n_0$.
Let $\SV\left(X\right)$ denote the set of all very Lipschitz bounded sequences from $C_b\left(X\right)$.
Define 
\[ \SV_\infty\left(X\right):= \SV\left(X\right)/\{\left(f_n\right)_{n=1}^\infty \in \SV\left(X\right) \mid \lim_{n\to\infty} \|f_n\| = 0\}. \]
\end{defn}
\bigskip

$\SV\left(X\right)$ is a $\mathrm C^*$-subalgebra of $l^\infty\left(\mathbb N,C_b\left(X\right)\right)$,\footnote{To check that the product of two very Lipschitz sequences is itself very Lipschitz, use the fact that if $f,g$ are bounded functions, such that $f$ is $L$-Lipschitz and $g$ is $L'$-Lipschitz, then $fg$ is $(\|f\|L'+\|g\|L)$-Lipschitz.}
and therefore the quotient $\SV_\oomega\left(X\right)$ is a C*-subalgebra of $(C_b(X))_\infty$.

E.g., if $X$ is a finitely generated group $G$ with the word metric, then $\SV_\oomega\left(X\right)$ can be identified with the fixed point algebra of $l^\infty\left(G\right)_\oomega$ under the action of $G$ induced by left translation on $l^\infty\left(G\right)$.
\bigskip

Recall the following definition from \cite{Roe:LectOnCoarseGeom}.

\begin{defn}
Let $X$ be a proper metric space.
A function $g \in C_b(X)$ is a \emph{Higson function} (also called a \emph{slowly oscillating} function) if, for every $R>0$ and $\e>0$, there exists a compact set $A \subseteq X$ such that for $x,y \in X \setminus A$, if $d(x,y) < R$ then $|g(x)-g(y)|<\e$.
The set of all Higson functions on $X$ is denoted $C_h(X)$.
\end{defn}
\bigskip

E.g., if $X$ is a finitely generated group $G$ with the word metric, then $C_h(X)\subseteq l^\infty(X)$ is the preimage of the fixed point algebra of $l^\infty\left(G\right)/c_0(G)$ under the action of $G$ induced by left translation on $l^\infty\left(G\right)$.
\bigskip

In the following, $\mathcal H$ is an $X$-module, and we view both $\SV_\oomega\left(X\right)$ and $B \subseteq \mathcal B\left(\mathcal H\right)$ as C*-subalgebras of $\mathcal B\left(\mathcal H\right)_\oomega$, and consider the relative commutant
\[ B \cap \SV_\oomega\left(X\right)'. \]
It is easy to see (at least in the standard cases of Example \ref{ex:Setup}) that any finite propogation operator commutes with $\SV_\oomega\left(X\right)$, and by taking limits it follows that
\[ \mathrm{Roe}(X,B) \subseteq B \cap \SV_\oomega\left(X\right)'. \]

The main result is as follows.
Recall that straight finite decomposition complexity, as introduced in \cite{DranishnikovZarichnyi:HaversC}, is a weakening of finite asymptotic dimension (\cite[Theorem 4.1]{GTY:GGD}). Both properties are defined in the following subsection.

\begin{thm}
\label{thm:MainThm}
Let $X$ be a proper metric space, $\mathcal H$ an $X$-module, and let $B \subseteq \mathcal B(\mathcal H)$ be a C*-subalgebra such that $C_b(X)BC_b(X)=B$, which is closed under block cutdowns, and such that \eqref{eq:ideal-cond} holds.
For $b \in B$, the following are equivalent.
\begin{enumerate}
\item $[b,f] = 0$ for all $f \in \SV_\oomega(X)$;
\item $b$ is quasi-local (it has finite $\e$-propogation for every $\e>0$);
\item $[b,g] \in \mathcal K(X,B)$ (i.e., $b$ essentially commutes with $g$) for all $g \in C_h(X)$.
\end{enumerate}
If $X$ has straight finite decomposition complexity, then these are also equivalent to
\begin{enumerate}
\item[(iv)] $b \in \mathrm{Roe}(X,B)$.
\end{enumerate}
\end{thm}
\bigskip

The equivalence of (i) and (ii) is fairly straightforward, and the equivalence of these conditions with (iii) (at least in the standard cases of Example \ref{ex:Setup}) seems to be known by coarse geometers; we shall provide a detailed proof for completeness.
The implication (iv) $\implies$ (ii) is straightforward and holds in complete generality.

The implication (i) $\Rightarrow$ (iv) was proven by Lange and Rabinovich for the uniform Roe algebra of $\mathbb Z^d$ (i.e., the case $X=\mathbb Z^d$, $\mathcal H=l^2(X)$, and $B=\mathcal B(\mathcal H)$ as in Example \ref{ex:Setup} (i)) in \cite{LangeRabinovich} (see \cite[Proposition 8]{RRS} for a proof in English).

The result (ii) $\Rightarrow$ (iv) was claimed by Roe in a remark on page 20 of \cite{Roe:CBMS} under a ``finite dimensionality'' assumption, but it was later found that his supposed proof was incomplete (\cite{Roe:Email}).
The present paper is to the authors' knowledge the first complete proof of a more general case (which is even more general than finite asymptotic dimension).

\begin{question}
Is there a uniformly discrete countable metric space with bounded geometry, for which (i)-(iii) does not imply (iv) of Theorem \ref{thm:MainThm}?
\end{question}
\bigskip

\subsection{Coarse geometric notions}
\label{sec:CoarseGeom}

We collect some terminology from \cite{GTY:Inv,GTY:GGD,DranishnikovZarichnyi:HaversC}.

\begin{defn}\label{def:sfdc}
  Let $X$ be a proper metric space, let $Z,Z' \subseteq X$, let $\X$ and $\Y$ be \emph{metric families} (i.e.\ at most countable sets of subsets of $X$), and finally let $R\geq0$.
  \begin{itemize}
  	\item We shall say that $\X$ is \emph{uniformly bounded}, if $\sup_{Y\in\X}\diam(Y) < \infty$.
  	\item We shall denote the \emph{metric neighbourhood} of $Z$ of radius $R$ by $N_R(Z):=\{z\in X\mid d(z,Z)\leq R\}$. We further set
	\[ N_R(\X) := \{N_R(Y): Y \in \X\}. \]
	\item The \emph{distance} between $Z$ and $Z'$ is $d(Z,Z') := \inf \{d(z,z'): z \in Z, z' \in Z'\}$.
	\item A family $(Y_j)_{j \in J}$ is \emph{$R$-disjoint} if $d(Y_j,Y_{j'}) > R$ for all $j\neq j'$; we write
	\[ \bigsqcup_{R\text{-disjoint}}Y_{j} \]
	for the union of the $Y_j$ to indicate that the family is $R$-disjoint.
  	\item We say that $Z$ \emph{$R$-decomposes over $\Y$}, if we can decompose $Z=X_0\cup X_1$ and
  	$$
  	X_i = \bigsqcup_{R\text{-disjoint}}X_{ij},\quad i=0,1,
  	$$
  	such that $X_{ij}\in \Y$ for all $i,j$.
  	\item We say that $\X$ \emph{$R$-decomposes over} $\Y$, denoted $\X\xrightarrow{\,R\,}\Y$, if every $Y\in\X$ $R$-decomposes over $\Y$.
  	\item We say that $X$ has \emph{asymptotic dimension at most $n$}, if for every $r\geq0$, we can decompose $X=X_0\cup \dots\cup X_n$ and
  	$$
  	X_i = \bigsqcup_{r\text{-disjoint}}X_{ij},\quad i=0,\dots,n,
  	$$
  	such that the metric family $\{X_{ij}\mid i,j\}$ is uniformly bounded.
  	\item We say that $X$ has \emph{straight finite decomposition complexity}, if for any sequence $0\leq R_1<R_2<\cdots$, there exists $m\in \N$ and metric families $\{X\} = \X_0, \X_1,\dots, \X_m$, such that $\X_{i-1}\xrightarrow{\,R_{i}\,}\X_i$ for $i=1,\dots, m$, and the family $\X_m$ is uniformly bounded.
  \end{itemize}
\end{defn}

The notion of straight finite decomposition complexity (sFDC) \cite{DranishnikovZarichnyi:HaversC} is apriori weaker than the original notion of finite decomposition complexity of Guentner, Tessera and Yu \cite{GTY:Inv,GTY:GGD}, see \ccite{Proposition 2.3}{DranishnikovZarichnyi:HaversC}. The definition of finite decomposition complexity uses a certain ``decomposition game'', which effectively means that the choices of $R_i$ can depend on the previous decompositions $\X_1,\dots,\X_{i-1}$.

Already finite decomposition complexity is weaker than finite asymptotic dimension (\cite[Theorem 4.1]{GTY:GGD}).

\section{Proof of (i) \texorpdfstring{$\Leftrightarrow$}{iff} (ii)}
\label{sec:IiffII}

To prove the main result, we begin with a technical-looking characterization of condition (ii).

\begin{lemma}
\label{lem:CommutantChar1}
Let $X$ be a proper metric space, let $\mathcal H$ be an $X$-module, and let $a \in \mathcal B(\mathcal H)$.
Then $\|[a,f]\|<\e$ for every $f \in \SV_\oomega\left(X\right)_1$ if and only if there exists $L>0$ such that $\|[a,f]\| < \e$ whenever $f \in C_b(X)_1$ is $L$-Lipschitz.
\end{lemma}

\begin{remark}
\label{rem:L-epsilon-notation}
As we shall need to refer to the conclusion of the above lemma later, we shall fix the following notation. In the setup as in the above lemma, we write $a\in \Leps(L,\e)$ if $\|[a,f]\| < \e$ whenever $f \in C_b(X)_1$ is $L$-Lipschitz.
\end{remark}

\begin{proof}
The reverse implication is immediate from the definition of $\SV\left(X\right)$.
For the forward direction, we use a proof by contradiction.
Suppose for a contradiction that, for every $n$ there exists $f_n \in C_b\left(X\right)_1$ that is $(1/n)$-Lipchitz and $\|[a,f_n]\| \geq \e$.

Then evidently, $\left(f_n\right)_{n=1}^\infty \in \SV\left(X\right)$ yet $\lim_\oomega \|[a,f_n]\| \geq \e$.
This contradicts the hypothesis that $\|[a,\SV_\oomega\left(X\right)_1]\| < \e$.
\end{proof}
\bigskip

\begin{proof}[Proof of Theorem \ref{thm:MainThm} (i) $\Rightarrow$ (ii)]
Suppose that $\left[b,\SV_\oomega\left(X\right)_1\right]=0$ and let $\e>0$.
By Lemma \ref{lem:CommutantChar1}, let $b \in \Leps(L,\e)$ (in the notation of Remark \ref{rem:L-epsilon-notation}) for some $L>0$.

We claim that $b$ has $\e$-propogation at most $L^{-1}$.
Certainly, suppose that $f,f' \in C_b\left(X\right)_1$ have $L^{-1}$-disjoint supports.
We may define $g \in C_b\left(X\right)$ such that $g|_{\mathrm{supp} f} \equiv 1$, $g|_{\mathrm{supp} f'} \equiv 0$ and $g$ is $L$-Lipschitz.
Hence, $\|[b,g]\| < \e$.
Consequently,
\[ \|fbf'\| = \|fgbf'\| \leq \|[b,g]\| + \|fbgf'\| < \e+0, \]
as required.
\end{proof}
\medskip

\begin{proof}[Proof of Theorem \ref{thm:MainThm} (ii) $\Rightarrow$ (i)]
Suppose that $b$ has finite $\e$-propogation for all $\e>0$.
Assume that $b$ is a contraction.
We shall verify the condition in Lemma \ref{lem:CommutantChar1}.
Therefore, let $\e>0$ be given.
Pick $N$ such that $6/N < \e/2$.
By the hypothesis, let $b$ have $(\e/\left(2N^2\right))$-propogation at most $R>0$.

Let $f \in C_b\left(X\right)_1$ be $(2RN)^{-1}$-Lipschitz.
We claim that $\|[b,f]\|<\e$.
Surely, define sets
\[ A_1 := f^{-1}\left([0,\tfrac1N]\right), \quad A_i := f^{-1}\left(\left(\tfrac{i-1}N,\tfrac iN\right]\right), \quad i=2,\dots,N. \]
These sets partition $X$ and, for $|i-j|>1$, $A_i$ is $(2R)$-disjoint from $A_j$.
We may find a partition of unity $e_1,\dots,e_N \in C_b(X)$ such that $e_i$ is supported in $N_{R/2}(A_i)$.
It follows that the supports of $e_i$ and $e_j$ are $R$-disjoint when $|i-j|>1$.

Thus,
\begin{align}
\label{eq:AiaAjSmall}
\|e_ib e_j\| &< \tfrac\e{2N^2}.
\end{align}
Also,
\begin{equation}
\label{eq:AiApproxf} f \approx_{1/N} \sum_{i=1}^N \tfrac iN e_i
\end{equation}
and so
\begin{eqnarray*}
\|[f,b]\| &\stackrel{\eqref{eq:AiApproxf}}\leq& \tfrac2N + \|[\sum_{i=1}^N \tfrac iN e_i,b]\| \\
&=& \tfrac2N + \left\| \left(\sum_{i=1}^N \tfrac iN e_ib\right)\left(\sum_{j=1}^N e_j\right) - 
\left(\sum_{i=1}^N e_i\right)\left(\sum_{j=1}^N \tfrac jN be_j\right)\right\| \\
&=& \tfrac2N + \left\| \sum_{i,j=1}^N \left(\tfrac iN - \tfrac jN\right) e_ibe_j\right\| \\
&\leq& \tfrac2N + \sum_{|i-j|>1} \|e_ibe_j\| + \left\| \sum_{|i-j|\leq 1} \left(\tfrac iN - \tfrac jN\right) e_ibe_j\right\|.
\end{eqnarray*}
The terms of the first sum are each dominated by $\frac\e{2{N^2}}$ by \eqref{eq:AiaAjSmall}, so this entire sum is less than $\e/2$.
The second sum can be broken into 4 sums with orthogonal terms (namely, note that when $i=j$, the terms vanish; what remains is $j=i+1$ and $j=i-1$, and we break each of these into even and odd parts).
Each of the terms of the second sum has norm at most $1/N$; thus, we have
\[
\|[f,b]\| < \tfrac2N + \tfrac\e2 + \tfrac4N < \e, \]
as required.
\end{proof}
\bigskip

\section{Proof of (i) \texorpdfstring{$\Rightarrow$}{implies} (iv)}
\label{sec:IimpliesIV}

In this section, we prove that (i) $\Rightarrow$ (iv) in Theorem \ref{thm:MainThm}.
We begin by establishing a few general functional analytic facts.

Recall that the \emph{strong* topology} on $\mathcal B(\mathcal H)$ is the one in which a net $(a_\alpha)$ converges to $a \in \mathcal B(\mathcal H)$ if and only if both $a_\alpha \to a$ and $a_\alpha^* \to a^*$ in the strong operator topology (i.e., $\|a_\alpha \xi - a\xi\| \to 0$ and $\|a_\alpha^* \xi - a^* \xi\| \to 0$ for every $\xi \in \mathcal H$).
A \emph{conditional expectation} from C*-algebra $A$ to a C*-subalgebra $B$ is a completely positive and contractive projection $E$ from $A$ to $B$ satisfying
\[ E(b_1ab_2) = b_1E(a)b_2 \]
for all $b_1,b_2 \in B$ and $a \in A$.

\begin{lemma}
\label{lem:CompactCE}
Let $\mathcal H$ be a Hilbert space and let $G$ be a subgroup of the group of unitary operators, which is compact in the strong$^*$ topology.
Then there is a unique conditional expectation $E_G:\mathcal B(\mathcal H) \to G'$ whose restriction to the unit ball is weak operator topology continuous.
It satisfies
\begin{equation}
\label{eq:CompactCE}
\left \|E_G(a)-a\right\| \leq \sup_{u \in G} \left\|[a,u]\right\|, \quad a \in \mathcal B(\mathcal H).
\end{equation}
\end{lemma}

\begin{proof}
Let $\mu_G$ be the normalized Haar measure on $G$ (under the strong$^*$ topology).
Fix $a \in \mathcal B(H)$, and consider the map $G \to \mathcal B(\mathcal H)$ defined by $u \mapsto u^*au$.
Then, with the strong$^*$ topology on the domain $G$ and the weak operator topology on the range $\mathcal B(\mathcal H)$, this map is continuous.
We may therefore integrate, defining
\[ E_G(a) := \scalebox{0.7}{WOT-}\hspace*{-1mm}\int_G u^*au\,d\mu_G(u). \]
(Here, $\scalebox{0.7}{WOT-}\hspace*{-1mm}\int_G \cdot\,d\mu_G$ indicates the Pettis integral, i.e., $E_G(a)$ is the unique operator satisfying
\[ \langle E_G(a)\xi, \eta\rangle = \int_G \langle u^*au\xi,\eta\rangle, \]
for all $\xi,\eta \in \mathcal H$.)
Using invariance of the Haar measure $\mu_G$, one easily sees that $E_G(a)$ commutes with all of $G$.

We now check \eqref{eq:CompactCE}; for this, set $\gamma:=\sup_{u \in G} \|[a,u]\|$.
For $\eta,\xi \in \mathcal H$,
\begin{align*}
\left|\left\langle \left(E_G(a)-a\right)\eta,\xi\right\rangle\right|
&= \left|\int_G \left\langle \left(u^*au-a\right)\eta,\xi\right\rangle\,d\mu_G(u)\right| \\
&\leq \int_G \left| \left\langle \left(u^*au-a\right)\eta,\xi\right\rangle\right|\,d\mu_G(u) \\
&\leq \int_G \left\|u^*au-a\right\|\,\left\|\eta\right\|\,\left\|\xi\right\|\,d\mu_G(u) \\
&= \int_G \left\|[u,a]\right\|\,\left\|\eta\right\|\,\left\|\xi\right\|\,d\mu_G(u) \\
&\leq \gamma \left\|\eta\right\|\,\left\|\xi\right\|.
\end{align*}
Therefore, \eqref{eq:CompactCE} follows.

In particular, we conclude that if $a \in G'$ then $E_G(a)=a$.
It is also straightforward to see that the function $E_G$ is c.p.c., and therefore it is a conditional expectation.

On the unit ball of $\mathcal B(\mathcal H)$, the integral defining $E_G$ can be uniformly approximated in the weak operator topology by (finite) Riemann sums, which themselves are continuous in the weak operator topology.
It follows that the restriction of $E_G$ to the unit ball is continuous using the weak operator topology.

If $E:\mathcal B(\mathcal H) \to G'$ is another conditional expectation whose restriction to the unit ball is weak operator topology continuous, then for a contraction $a \in \mathcal B(\mathcal H)$,
\begin{align*}
E_G(a)
&= E\left(E_G(a)\right) \quad &&\text{($E$ fixes $G'$)} \\
&= E\Big(\scalebox{0.7}{WOT}\text{-}\hspace*{-1mm}\int_G u^*au\,d\mu_G(u)\Big) && \\
&= \scalebox{0.7}{WOT}\text{-}\hspace*{-1mm}\int_G E(u^*au)\,d\mu_G(u) \ \ \  &&\text{(WOT-continuity of $E|_{\mathcal B(\mathcal H)_1}$)}\\
&= \scalebox{0.7}{WOT}\text{-}\hspace*{-1mm}\int_G u^*E(a)u\,d\mu_G(u) \ \ \ &&\text{($E$ is a conditional expectation)} \\
&= E_G(E(a)) && \\
&= E(a) \quad &&\text{($E_G$ fixes $G'$)}.
\end{align*}
Thus, $E=E_G$.
\end{proof}
\bigskip

Recall that an \emph{atomic} abelian von Neumann algebra is a von Neumann algebra isomorphic to $l^\infty(X)$, for some set $X$.
In the following, when $\mathcal H=l^2(X)$, then the conditional expectation $\mathcal B(l^2(X)) \to l^\infty(X)$ consists simply of taking an operator to its diagonal.

\begin{cor}
\label{cor:AtomicCE}
Let $D \subset \mathcal B(\mathcal H)$ be an atomic abelian von Neumann algebra.
Then there is a unique conditional expectation $E_D:\mathcal B(\mathcal H) \to D'$ whose restriction to the unit ball is weak operator topology continuous.
It satisfies
\begin{equation}
\label{eq:AtomicCE}
\left \|E_D(a)-a\right\| \leq \sup_{x \in D, \|x\| \leq 1} \left\|[a,x]\right\|, \quad a \in \mathcal B(\mathcal H).
\end{equation}
\end{cor}

\begin{proof}
Without loss of generality, $D$ contains the identity operator.
$D$ is generated by a family of orthogonal projections $(p_j)_{j\in J}$, whose sum converges strongly to $1$.
Define 
\[ G:=\left\{ \sum_{j\in J} (-1)^{\alpha_j} p_j : (\alpha_j)_{j \in J} \in (\mathbb Z/2)^{J}\right\}. \]
This is a strong$^*$ compact subgroup of the unitary group of $D$ (it is homeomorphic to $(\mathbb Z/2)^{J}$ with the product topology), so that Lemma \ref{lem:CompactCE} applies to it.
It is clear that it generates $D$ as a von Neumann algebra, so that $G'=D'$.
The conclusion follows from Lemma \ref{lem:CompactCE}.
\end{proof}
\bigskip

\begin{cor}
\label{cor:CommuteCEEst}
Let $X$ be a proper metric space, let $\mathcal H$ an $X$-module, and let $a \in \mathcal B\left(\mathcal H\right)$.
Suppose $a \in \Leps(L,\e)$ for some $L,\e>0$ (in the notation of Remark \ref{rem:L-epsilon-notation}).
Let $(e_j)_{j \in J}$ be a family of positive contractions from $C_b(X)$ with $(2L^{-1})$-disjoint supports, and define $e:=\sum_{j\in J} e_j$.
Then, with $\theta_{(e_j)_{j\in J}}$ from Definition \ref{def:BlockCutdown}, 
\[ \|eae-\theta_{(e_j)_{j\in J}}(a)\| \leq \e. \]
\end{cor}

\begin{proof}
Set $A_j$ equal to the support of $e_j$ for each $j \in J$.
We may find pairwise disjoint projections $p_j \in \mathcal B(\mathcal H)$, for $j \in J$, such that $p_j$ acts as a unit on $e_j$.
Define $D$ to be the von Neumann subalgebra generated by $\{p_j:j \in J\}$ (with unit $1_D=\sum_j p_j$), and let $E_D:\mathcal B(\mathcal H) \to D'$ be the unique conditional expectation provided by Corollary \ref{cor:AtomicCE}.
Then one finds that for $x \in \mathcal B(\mathcal H)$, $E_D(x) = \sum_{j\in J} p_jxp_j$ (converging in the strong operator topology), and therefore
\[ E_D(eae) = \theta_{(e_j)_{j \in J}}(a). \]

Using $(2L^{-1})$-disjointness of the family $\left(A_j\right)_{j\in J}$, for $f \in D_1$, there exists a function $\tilde f \in C_b\left(X\right)_1$ that is $L$-Lipschitz such that
\[ f = 1_D\tilde f. \]
Therefore,
\begin{align*}
f(eae)
&= \tilde feae \\
&= e\tilde f a e \\
&\approx_\e e a \tilde f e \\
&= (eae) f.
\end{align*}
Hence, by Corollary \ref{cor:AtomicCE}, 
\[ e ae \approx_\e E_D(e ae) = \theta_{(e_j)_{j\in J}}(a), \]
as required.
\end{proof}
\bigskip

\begin{defn}
Let $X$ be a proper metric space, $\mathcal H$ be an $X$-module, $a\in\mathcal B(\mathcal H)$ and let $\X$ be a metric family (of subsets of $X$). We say that $a$ is \emph{block diagonal with respect to $\X$}, if there exists an equicontinuous family $(e_j)_{j\in J}$ of positive contractions in $C_b(X)$ with pairwise disjoint supports, such that $a = \theta_{(e_j)_{j\in J}}(a)$, and the support of each $e_j$ is contained in some set $Y_j\in\X$. Furthermore, in this case we shall denote $a_{Y_j}:=e_jae_j$ and call these operators \emph{blocks} of $a$.
\end{defn}

The next lemma sets up the ``induction step'' to be applied in the context of the decomposition game, in the proof of Theorem \ref{thm:MainThm} (i) $\Rightarrow$ (iv).

\begin{lemma}
\label{lem:FDCInductionStepBasic}
Let $X$ be a proper metric space, and let $\Y$ be a metric family such that $\{X\} \xrightarrow{\,4L^{-1}+4\,}\Y$ for some $L>0$. Let $\mathcal H$ be an $X$-module, and let $a \in \mathcal B\left(\mathcal H\right)$.
Let $\e>0$ be such that $a\in\Leps(L,\e)$. Then we can write
\[
a \approx_{8\e} a_{00}+a_{01}+a_{10}+a_{11},
\]
where each $a_{ii'}$ is of the form $\theta_{(f_k)_{k\in K}}(gag')$ for some contractions $g,g' \in C_b(X)$ and some family $(f_k)_{k\in K}$ of $1$-Lipschitz positive contractions in $C_b(X)$ with disjoint supports, such that the support of each $f_k$ is contained in a set in $N_{L^{-1}+1}(\Y)$. 
\end{lemma}

\begin{proof}
By the decomposition assumption, we can write
$$
X=X^{(0)}\cup X^{(1)},\quad X^{(i)}=\bigsqcup_{\substack{j\in J_i\\(4L^{-1}+4)\text{-disjoint}}}X_j^{(i)},\quad i=0,1,
$$
with $X_j^{(i)}\in\Y$ for each $i,j$.
We may find a partition of unity consisting of $1$-Lipschitz positive contractions $e^{(i)}_j \in C_b(X)$, for $i=0,1$ and $j \in J_i$, such that the support of $e^{(i)}_j$ is contained in $N_1(X^{(i)}_j)$.
It follows that for each $i$, the supports of $(e^{(i)}_j)_{j \in J_i}$ are $(4L^{-1}+2)$-disjoint.

For each $i=0,1$, define $e^{(i)}:=\sum_{j \in J_i} e^{(i)}_j$.
Since
$$
a = e^{(0)}ae^{(0)} + e^{(0)}ae^{(1)} + e^{(1)}ae^{(0)} + e^{(1)}ae^{(1)},
$$
it suffices to find $a_{ii'} \approx e^{(i)}ae^{(i')}$ for each $i,i'\in\{0,1\}$ with the required properties.
(We will be precise about the degree of approximation -- in short, it depends on whether $i$ and $i'$ are equal.)

For the case $i=i'$, Corollary \ref{cor:CommuteCEEst} shows that 
\[ e^{(i)}ae^{(i)} \approx_\e \theta_{(e^{(i)}_j)_{j\in J_i}}(e^{(i)}ae^{(i)}) =: a_{ii}. \]
 The latter operator is clearly block diagonal with respect to $N_1(\Y)$ (hence also with respect to $N_{L^{-1}+1}(\Y)$).

Turning now to the case $i\not=i'$,
note that for fixed $i$, the family $\left(N_{L^{-1}+1}(X^{\left(i\right)}_j)\right)_{j\in J_i}$ is $(2L^{-1}+2)$-disjoint.
For each $i,j$, there exists $\hat{e}^{\left(i\right)}_j \in C_b\left(X\right)$ that is $L$-Lipschitz, that acts as the identity on $e^{(i)}_j$, and is supported on $N_{L^{-1}+1}(X^{\left(i\right)}_j)$.
For each $i$, define $\hat{e}^{(i)}:=\sum_{j\in J_i} \hat{e}^{(i)}_j$.
We have
\begin{align*}
e^{(i)}ae^{(i')}
&= e^{(i)}\hat{e}^{(i)}a\hat{e}^{(i')}e^{(i')} \\
&\approx_{2\e} e^{(i)}\hat{e}^{(i')}a\hat{e}^{(i)}e^{(i')}.
\end{align*}
For each $j\in J_i$ and $j'\in J_{i'}$, there exists a $1$-Lipschitz positive contraction $f_{j,j'} \in C_b(X)$ that is $1$ on $N_{L^{-1}+1}(X^{\left(i\right)}_j) \cap N_{L^{-1}+1}(X^{\left(i'\right)}_{j'})$, and is supported on the metric neighbourhood of this set of radius $1$.
In particular, the support of each $f_{j,j'}$ is contained in a set in $N_{L^{-1}+2}(\Y)$, the family of supports of the family $(f_{j,j'})_{j,j'\in J_{i'}}$ is $(2L^{-1})$-disjoint, and $f:=\sum_{j,j'} f_{j,j'}$ acts as an identity on (both sides of) $e^{(i)}\hat{e}^{(i')}a\hat{e}^{(i)}e^{(i')}$.
Applying Corollary \ref{cor:CommuteCEEst}, we obtain
\begin{align*}
fe^{(i)}\hat{e}^{(i')}a\hat{e}^{(i)}e^{(i')}f
\approx_{\e} \theta_{(f_{j,j'})_{j\in J_i,\,j'\in J_{i'}}}\left(e^{(i)}\hat{e}^{(i')}a\hat{e}^{(i)}e^{(i')}\right).
\end{align*}
Thus,
\begin{align*}
e^{(i)}ae^{(i')}
&\approx_{2\e} e^{(i)}\hat{e}^{(i')}a\hat{e}^{(i)}e^{(i')} \\
&= fe^{(i)}\hat{e}^{(i')}a\hat{e}^{(i)}e^{(i')}f \\
&\approx_{\e} \theta_{(f_{j,j'})_{j\in J_i,\,j'\in J_{i'}}}\left(e^{(i)}\hat{e}^{(i')}a\hat{e}^{(i)}e^{(i')}\right) =: a_{ii'}.
\end{align*}
By construction, it is clear that $a_{ii'}$ is block diagonal with respect to $N_{L^{-1}+2}(\Y)$.

Summarizing, we have $a\approx_{\e+3\e+3\e+\e}a_{00}+a_{10}+a_{01}+a_{11}$, and all the $a_{ii'}$ are of the right form.
\end{proof}

We now strengthen the previous lemma by allowing an arbitrary metric family in place of $\{X\}$ and with $a$ a correspondingly block diagonal operator.

\begin{lemma}
\label{lem:FDCInductionStep}
Let $X$ be a proper metric space, and let $\X$ and $\Y$ be metric families, such that $\X\xrightarrow{\,4L^{-1}+4\,}\Y$ for some $L>0$. Let $\mathcal H$ be an $X$-module, and let $a \in \mathcal B\left(\mathcal H\right)$ be block diagonal with respect to $\X$. Let $\e>0$ be such that $a\in\Leps(L,\e)$. Then we can write
\begin{equation}
\label{eq:FDCInductionStepApprox}
a \approx_{8\e} a_{00}+a_{01}+a_{10}+a_{11},
\end{equation}
where each $a_{ii'}$ is of the form $\theta_{(f_k)_{k\in K}}(gag')$ for some contractions $g,g' \in C_b(X)$ and some equicontinuous family $(f_k)_{k\in K}$ of positive contractions in $C_b(X)$ with disjoint supports, such that the support of each $f_k$ is contained in a set in $N_{L^{-1}+1}(\Y)$. 
In particular:
\begin{enumerate}
\item each $a_{ii'}$ is block diagonal with respect to $N_{L^{-1}+1}(\Y)$,
\item if $a\in\Leps(L',\e')$ for some $L',\e'>0$, then each $a_{ii'}$ is in $\Leps(L',\e')$ as well, and
\item if $B \subseteq \mathcal B(\mathcal H)$ is a C*-subalgebra such that $C_b(X)BC_b(X)=B$ and $B$ is closed under block cutdowns, and if $a$ is in $B$, then each $a_{ii'}$ is in $B$ as well.
\end{enumerate}
\end{lemma}

\begin{proof}
Without loss of generality, both $\X$ and $\Y$ are closed under taking subsets.
Start by letting $(e_j)$ be an equicontinuous family of positive contractions in $C_b(X)$ with disjoint supports, such that $Y_j:=\mathrm{supp}(e_j) \in \X$.
Applying Lemma \ref{lem:FDCInductionStepBasic} to each $a_{Y_j}$ yields
\[ a_{Y_j} \approx_{8\e} a^j_{00}+a^j_{01}+a^j_{10}+a^j_{11}, \]
satisfying the conclusions of that lemma.
Set
\[ a_{ii'} := \sum_j a^j_{ii'}. \]
Then \eqref{eq:FDCInductionStepApprox} follows from \eqref{eq:BlockNorm}

To see that each $a_{ii'}$ has the right form, fix $i$ and $i'$.
For each $j$, there exist contractions $g_j,g_j' \in C_b(X)$ and a family $(f_{j,k})_{k\in K_j}$ of $1$-Lipschitz positive contractions in $C_b(Y_j)$ with disjoint supports, such that
\[ a^j_{ii'} = \theta_{(f_{j,k})_{k\in K_j}}(g_jag_j'). \]
and the support of each $f_{j,k}$ is of the form $Y \cap Z$ for some $Z \in N_{L^{-1}+1}(\Y)$.

Then observe that 
\[ a_{ii'} = \theta_{(e_j^{1/4}f_{j,k})_{j \in J, k \in K_j}}((\sum_j e_j^{1/4}g_j)a(\sum_j e_j^{1/4}g_j)), \]
where the family appearing in this block-cutdown formula, namely
\[ (e_j^{1/4}f_{j,k})_{j \in J, k \in K_j} \]
is equicontinuous and contained in $N_{L^{-1}+1}(\Y)$.

(i)-(ii) are immediate consequences of the form that $a_{ii'}$ takes.
Since multiplication by $C_b(X)$ preserves block structure, and using \eqref{eq:BlockNorm}, (iii) can also be seen to be a consequence of the form that $a_{ii'}$ takes.
\end{proof}

We have seen in the previous lemma that we will need to work with ``thickened'' metric families, so we record the following straightforward observation.

\begin{lemma}
\label{lem:ThickenedDecomposes}
Let $\X$ and $\Y$ be metric families, $R,S\geq 0$. Assume that $\X\xrightarrow{\,R\,}\Y$ and that $R-2S\geq 0$. Then $N_S(\X)\xrightarrow{\,R-2S\,}N_S(\Y)$.
\end{lemma}

\begin{proof}[Proof of Theorem \ref{thm:MainThm} (i) $\Rightarrow$ (iv)]
Recall that we are given an operator $b$ on an $X$-module $\mathcal H$ satisfying $[b,f]=0$ for all $f\in \SV_\oomega(X)$.
Given $\e>0$, our task is to produce a finite propogation operator in $B$ which is $\e$-far from $b$. Lemma \ref{lem:CommutantChar1} provides us with $L_n$ for every 
\[ \e_n:=\e/(2\cdot8^n), \]
such that $b\in\Leps(L_n,\e_n)$. Let 
\[ R_n:=4(L_n^{-1}+1)+2(L_{n-1}^{-1}+1)+\dots+2(L_1^{-1}+1). \]
As $X$ has straight finite decomposition complexity (see Definition \ref{def:sfdc}), we obtain metric families $\{X\}=\X_0,\X_1,\dots,\X_m$, such that $\X_{n-1}\xrightarrow{\,R_n\,}\X_{n}$ for $n\in\{1,\dots,m\}$ and $\X_m$ is uniformly bounded.
Note that Lemma \ref{lem:ThickenedDecomposes} gives us that
$$
N_{(L_{n-1}^{-1}+1)+\dots+(L_1^{-1}+1)}(\X_{n-1})
\xrightarrow{\,4(L_n^{-1}+1)\,}
N_{(L_{n}^{-1}+1)+\dots+(L_1^{-1}+1)}(\X_{n}).
$$
Thus, we can inductively apply Lemma \ref{lem:FDCInductionStep}, with $L_n$, $\e_n$, the operators obtained in the previous iteration, and metric families from the above display. After $m$ steps, we will have approximated the operator $b$ by an operator $b'$ which is a sum of finitely many ($4^m$ to be precise) operators in $B$ which are block diagonal with respect to 
\[ N_{(L_m^{-1}+1)+\dots+(L_1^{-1}+1)}(\X_m). \]
Since $\X_m$ is uniformly bounded, so is the above family, and therefore operators which are block diagonal with respect to it have finite propogation; consequently, $b' \in \mathrm{Roe}(B,X)$.
Tracing through the estimates given by Lemma \ref{lem:FDCInductionStep}, we compute that the distance from $b$ to $b'$ is at most 
\begin{align*}
	8\e_1 + 4\left(8\e_2 + 4\left(8\e_3 + 4\left(\dots\right) \right) \right)
	= \e\left(\tfrac12+\tfrac14+\tfrac18+\dots\right) = \e.
\end{align*}
This finishes the proof.
\end{proof}

\begin{remark}
When the asymptotic dimension of $X$ is at most $d<\infty$, the induction component of the above proof can be removed: one can use the idea of Lemma \ref{lem:FDCInductionStepBasic} with a decomposition of $X$ into $d+1$ (instead of 2) uniformly bounded, $(4L^{-1}+4)$-disjoint families, and correspondingly approximate $a$ be a sum of $(d+1)^2$ block diagonal operators.
\end{remark}

\bigskip

\section{Higson functions}
\label{sec:Higson}

To prepare to prove (i) $\Leftrightarrow$ (iii) of Theorem \ref{thm:MainThm}, we begin by considering a special class of Higson functions which are more closely related to our definition of $\SV_\oomega(X)$.

\begin{defn}
Let $X$ be a proper metric space.
A function $g \in C_b(X)$ is a \emph{Lipschitz--Higson function} if, for every $L>0$, there exists a compact set $A \subseteq X$ such that $g|_{X \setminus A}$ is $L$-Lipschitz.
The set of all Lipschitz--Higson functions on $X$ is denoted $C_{lh}(X)$.
\end{defn}
\bigskip

Fix a proper metric space $X$ and a point $x_0 \in X$.
For $R>0$, define $e_R \in C_0(X)$ by
\begin{equation}
\label{eq:e_Rdef}
e_R(x) := \max\{0, 1-d(x,B_R(x_0))/R\}.
\end{equation}
Observe that $e_R$ is $R^{-1}$-Lipschitz, is $1$ on $\bar{B}_R(x_0)$, and vanishes outside of $B_{2R}(x_0)$.
\bigskip

\begin{lemma}
\label{lem:LipHig}
$C_h(X) = C_{lh}(X) + C_0(X)$.
\end{lemma}

\begin{proof}
The inclusion $\supseteq$ is straightforward.
To go the other direction, let $g \in C_h(X)$.
We shall produce $f \in C_{lh}(X)$ such that $f-g \in C_0(X)$.
Without loss of generality, $g$ is a positive contraction.
Fix a point $x_0 \in X$.

Recursively define $R_0:=0$ and $R_n \geq \max\{2(n+1),2R_{n-1}\}$ such that if $x,y \in X \setminus B_{R_n}(x_0)$ and $d(x,y) < n+1$ then $|g(x)-g(y)| < \frac1{2(n+1)}$.
Using $e_R$ from \eqref{eq:e_Rdef}, set 
\[ g_1:= e_{R_1}g, \quad g_n := \left(e_{R_n}-e_{R_{n-1}}\right)g, \quad n\geq 2. \]
Note that, for $n \geq 2$, $g_n$ is such that if $d(x,y) < n$ then $|g_n(x)-g_n(y)|<\frac1n$, and that
\begin{equation}
\label{eq:gSum} g = \sum_{n=1}^\infty g_n,
\end{equation}
converging pointwise (as at each point, at most two terms of the sum are nonzero).

Define $f_1:=g_1$.

Fix $n\geq 2$; we shall define a function $f_n$ which approximates $g_n$, but is more Lipschitz.
Define
\[ A_i := g_n^{-1}\left[\tfrac i{n},1\right], \quad i=1,\dots,n, \]
and define $c_i\in C_b(X)$ by
\[ c_i(x):= \max\left\{1-\tfrac{d(x,A_i)}n,0\right\}, \quad i=1,\dots,n, \]
which is $(1/n)$-Lipschitz.
Set
\[ f_n:= \frac1n \sum_{i=1}^n c_i, \]
which is also $(1/n)$-Lipschitz, as it is an average of such.
Moreover, $\|f_n-g_n\|\leq \frac1n$, and the support of $f_n$ is contained in the support of $g_n$.

Set $f:=\sum_{n=1}^\infty f_n$; as in \eqref{eq:gSum}, at each point, at most two terms of the sum are nonzero.
Using this fact, one sees that $\sum_{n=n_0}^\infty f_n$ is $2n_0^{-1}$-Lipschitz, for all $n_0 \geq 2$.
Moreover, $f$ agrees with this tail sum outside of $B_{2R_{n_0-1}}$, which proves that $f \in C_{lh}(X)$.

Similarly, since $f-g$ agrees with the tail $\sum_{n=n_0}^\infty (f_n-g_n)$ outside of $B_{2R_{n_0-1}}$, and this tail has norm at most $\frac2n$, it follows that $f-g \in C_0(X)$.
\end{proof}
\bigskip

\begin{remark}
We note in passing that, due to the previous lemma, the Higson corona $\nu X$ (defined as the compact Hausdorff space satisfying $C(\nu X) \cong C_h(X)/C_0(X)$) satisfies
\[ C(\nu X) \cong C_{lh}(X)/C_0(X). \]
\end{remark}
\bigskip

Now we set out two constructions to be used, producing a Lipschitz--Higson function from a very Lipschitz sequence and vice versa.
Neither construction is canonical: both depend on a number of choices.
\bigskip

Let $(f_k)_{k=1}^\infty \in \SV(X)$, let $(f_{k_i})_{i=1}^\infty$ be a subsequence, and let $(R_i)_{i=0}^\infty \subset (0,\infty)$ be a sequence such that $R_{i+1} \geq 6R_i$ for each $i$.
From these, and using $e_R$ from \eqref{eq:e_Rdef}, define
\begin{equation}
\label{eq:HigsongDef}
 g_{x_0,(f_{k_i})_{i=1}^\infty,(R_i)_{i=0}^\infty} := \sum_{i=1}^\infty f_{k_i}(e_{R_i}-e_{3R_{i-1}}) \in C_b(X).
\end{equation}
Note that the supports of the functions in the summation are pairwise disjoint, so we can treat the series as converging pointwise.
It is straightforward to see that $\|g_{x_0,(f_{k_i})_{i=1}^\infty,(R_i)_{i=0}^\infty}\| \leq \|(f_k)_{k=1}^\infty\|$.

\begin{lemma}
\label{lem:Higsong}
Fix a proper metric space $X$ and a point $x_0 \in X$.
For $(f_k)_{k=1}^\infty \in \SV(X)$, a subsequence $(f_{k_i})_{i=1}^\infty$, and a sequence $(R_i)_{i=0}^\infty \subset (0,\infty)$ such that $R_{i+1} \geq 6R_i$ for each $i$, let $g_{x_0,(f_{k_i})_{i=1}^\infty,(R_i)_{i=0}^\infty}$ be as defined in \eqref{eq:HigsongDef}.
Then
\[ g_{x_0,(f_{k_i})_{i=1}^\infty,(R_i)_{i=0}^\infty} \in C_{lh}(X). \]
\end{lemma}

\begin{proof}
Without loss of generality, we may assume $\|f_k\| \leq 1$ for all $k$.
For ease of notation, we set
\[ g:=g_{x_0,(f_{k_i})_{i=1}^\infty,(R_i)_{i=0}^\infty}. \]

As in the definition of a Lipschitz--Higson function, let $L>0$ be given.
Pick $i_0$ such that $R_{i_0-1} > 2/L$ and such that $f_i$ is $(L/2)$-Lipschitz for all $i \geq i_0$.
We will be done when we show that the restriction of $g$ to $X \setminus B_{R_{i_0}}(x_0)$ is $L$-Lipschitz.

For $i \geq i_0$, $e_{R_i}-e_{3R_{i-1}}$ is $(L/2)$-Lipschitz, so that the product
\[ h_i := f_{k_i}(e_{R_i}-e_{3R_{i-1}}) \]
is $L$-Lipschitz.
For $x,y \in X \setminus B_{R_{i_0}}(x_0)$, note that, by the definition of $e_R$ and the condition $R_{i+1} \geq 6R_i$, that at least one of the following conditions holds.
\begin{enumerate}
\item $d(x,y) \geq R_{i_0}$, or
\item There exists $i$ such that $g(x)=h_i(x)$ and $g(y)=h_i(y)$.
\end{enumerate}

In the first case,
\[ |g(x)-g(y)| \leq 2\|g\| = 2 \leq 3LR_{i_0-1} \leq 3LR_{i_0} \leq Ld(x,y). \]

In the second case, since $h_i$ is $L$-Lipschitz, it follows that $|g(x)-g(y)|\leq Ld(x,y)$.
\end{proof}
\bigskip

Next, let $g \in C_h(X)$ be given, along with a sequence $(R_k) \in (0,\infty)$ such that $\lim_{k\to\infty} R_k = \infty$.
From this data, define (using $e_R$ from \eqref{eq:e_Rdef})
\begin{equation}
\label{eq:SVFdef}
F_{x_0,g,(R_k)_{k=1}^\infty} := (f_k)_{k=1}^\infty \in l^\infty(\mathbb N,l^\infty(X)) \ \text{where} \ f_k := (1-e_{R_k})g.
\end{equation}

\begin{lemma}
\label{lem:SVF}
Fix a proper metric space $X$ and a point $x_0 \in X$.
For a Higson function $g \in C_{lh}(X)$ and a sequence $(R_k) \in (0,\infty)$ such that $\lim_{k\to\infty} R_k = \infty$, let $F_{x_0,g,(R_k)_{k=1}^\infty}$ be as defined in \eqref{eq:SVFdef}.
Then
\[ F_{x_0,g,(R_k)_{k=1}^\infty} \in \SV(X). \]
\end{lemma}

\begin{proof}
Without loss of generality, we may assume $\|g\| \leq 1$.
As in the definition of $\SV(X)$, let $L>0$ be given.
Pick $M \geq 2/L$ such that $g$ is $(L/2)$-Lipschitz on $X \setminus B_M(x_0)$.
Pick $k_0$ such that $R_k \geq M$ for all $k \geq k_0$.
For $k \geq k_0$, using the fact that $(1-e_{R_k})$ is $(L/2)$-Lipschitz and vanishes on $B_M(x_0)$, it is easy to see that $f_k=g(1-e_{R_k})$ is $L$-Lipschitz.
\end{proof}
\bigskip

\begin{proof}[Proof of Theorem \ref{thm:MainThm} (i) $\Rightarrow$ (iii)]
By Lemma \ref{lem:LipHig}, (iii) is equivalent to $[g,b] \in \mathcal K(X,B)$ for all $g \in C_{lh}(X)$, which is the statement we will prove assuming (i).
Assume that $[b,g] \not\in \mathcal K(X,B)$ for some $g \in C_{lh}(X)$.
Set $\e:=\|[b,g]+\mathcal K(X,B)\|$.
Fix a point $x_0 \in X$, let $(R_k)_{i=1}^\infty \subset (0,\infty)$ be any sequence such that $\lim_{i\to\infty} R_i = \infty$, and define $(f_k)_{k=1}^\infty := F_{x_0,g,(R_k)_{k=1}^\infty}$ as in \eqref{eq:SVFdef}, that is,
\[ f_k := (1-e_{R_k})g. \]
By Lemma \ref{lem:SVF}, $(f_k)_{k=1}^\infty \in \SV(X)$.
Since $e_{R_k}\in C_0(X)$ for each $k$, using the condition \eqref{eq:ideal-cond} we obtain
\[ [f_k,b] = [g,b] - [e_{R_k}g,b] \in [g,b] + \mathcal K(X,B) \]
and therefore by the definition of $\e$, 
\[ \|[f_k,b]\| \geq \e. \]
Consequently, $\lim_{k\to \oomega} \|[f_k,b]\| \geq \e$, so that $b$ does not commute with the image of $(f_k)_{k=1}^\infty$ in $\SV_\oomega(X)$.
\end{proof}
\medskip

Before we embark on the proof of Theorem \ref{thm:MainThm} (iii) $\implies$ (i), note that as a consequence of the Spectral Theorem, we may extend the $X$-module structure on $\mathcal{B}(\mathcal{H})$ from bounded continuous functions to bounded Borel functions on $X$.
This is convenient in the following proof, since it allows us to easily ``cut up'' operators on $\mathcal{H}$ using characteristic functions of Borel sets in $X$.
(Of course, we cannot assume the algebra $B$ is closed under multiplication by these bounded Borel functions.)
We opt for this approach for the sake of readability, although it is possible to modify the proof to only use continuous functions for the price of more approximations.

\begin{proof}[Proof of Theorem \ref{thm:MainThm} (iii) $\Rightarrow$ (i)]
Since each of $\SV_\oomega(X)$, $C_{h}(X)$, and $\mathcal K(X,B)$ is $^*$-closed, it suffices to prove the theorem in the case that $b$ is self-adjoint.
We henceforth assume that $b=b^*$.
Fix a point $x_0 \in X$, and to shorten notation in this proof, set
\[ B_R:= B_R(x_0). \]
For each $R>0$, we will use $\chi_{B_R}$ to denote the support projection of a function whose cozero set is $B_R$.

Assume that $[b,f] \neq 0$ for some $f \in \SV_\oomega(X)$.
Then in fact, $[b,f] \neq 0$ for some $f=(f_k)_{k=1}^\infty$ for which each $f_k$ is a self-adjoint contraction; we fix this sequence.
Let $0 < \e <\|[b,f]\|$.

Consider now two cases.

\textit{Case 1}. There exists $R_0>0$ such that for all $S>0$ there are infinitely many $k$ for which
\[ \|\chi_{B_{R_0}}[b,f_k](1-\chi_{B_S})\| > \frac\e5. \]

Roughly, for this case, we will construct some $g$ (of the form \eqref{eq:HigsongDef}) such that the block-column of $[b,g]$ corresponding to $B_{R_0}$ doesn't converge to $0$ at $\infty$.

Note that as $k \to \infty$, $f_k|_{B_{R_0}}$ tends towards being constant; so without loss of generality, we may assume that $f_k|_{B_{R_0}}$ is constant.
Adding a scalar to each $f_k$, we arrive at another sequence $\left(\hat f_k\right)_{n=1}^\infty$ with the same properties as $(f_k)_{n=1}^\infty$ (that is, self-adjoint and satisfying the Case 1 condition), such that $\hat f_k|_{B_{R_0}} \equiv 0$.
From this it follows that $\chi_{B_{R_0}}[b,\hat f_k] = \chi_{B_{R_0}}b\hat f_k$ for all $n$, so that we have:
for all $S>0$ there exist infinitely many $k$ such that
\[ \|\chi_{B_{R_0}}b\hat f_k(1-\chi_{B_S})\| > \frac\e5. \]

Using $R_0$ as above and $k_0:=0$, recursively choose $k_1 < k_2 < \cdots$ and $R_1 \geq 6R_0$, $R_2 \geq 6R_1$, $\dots$ as follows.
Having chosen $R_{i-1}$, pick $k_i>k_{i-1}$ such that
\[ \|\chi_{B_{R_0}}a\hat f_{k_i}(1-\chi_{B_{6R_{i-1}}})\| > \frac\e5 \quad\text{and}\quad
  \|\hat f_{k_i}\chi_{B_{6R_{i-1}}}\|\leq \frac{\e}{10\|b\|}.\]
(The second inequality can be arranged as $\hat f_{k}$ converge to $0$ on any given bounded subset of $X$.)
Then, since $(e_R)_{R=1}^\infty$ converges strongly to $1$, pick $R_i \geq 6R_{i-1}$ such that
\[ \|\chi_{B_{R_0}}b\hat f_{k_i}(1-\chi_{B_{6R_{i-1}}})e_{R_i}\|>\frac\e5. \]
Since $e_{R_i}-e_{3R_{i-1}}$ differs from $(1-\chi_{B_{6R_{i-1}}})e_{R_i}$ only on $B_{6R_{i-1}}$,
it follows that
\[ \|\chi_{B_{R_0}}b\hat f_{k_i}(e_{R_i}-e_{3R_{i-1}})\|>\frac\e{10}. \]

Using these recursive choices, define $g:=g_{x_0,(\hat f_{k_i})_{i=1}^\infty,(R_i)_{i=0}^\infty} \in C_{lh}(X) \subseteq C_h(X)$ by Lemma \ref{lem:Higsong}.
If $[b,g] \in \mathcal K(X,B)$, then $\|[b,g](1-\chi_{B_S})\|\to 0$ as $S \to \infty$.
However, given $S \geq R_0$, there exists $i$ such that $3R_{i-1} > S$.
Then
\begin{align*}
\|[b,g](1-\chi_{B_S})\|
&\geq \|\chi_{B_{R_0}}[b,g](1-\chi_{B_S})(\chi_{B_{2R_i}} - \chi_{B_{3R_{i-1}}})\| \\
&= \|\chi_{B_{R_0}}b\hat f_{k_i}(e_{R_i}-e_{3R_{i-1}})\| \\
&> \frac\e{10},
\end{align*}
which is a contradiction.
This concludes the proof in Case 1.
\medskip

\textit{Case 2}. For every $R>0$, there exists $S>0$ such that, for all but finitely many $k \in \mathbb N$,
\[ \|\chi_{B_R}[b,f_k](1-\chi_{B_S})\| \leq \frac\e5. \]
Without loss of generality, we may assume that $S>R$.

Roughly, for this case, we will construct some $g$ (of the form \eqref{eq:HigsongDef}) such that the certain blocks on the diagonal of $[b,g]$ don't converge to $0$ at $\infty$.

In preparation for this, suppose we are given $R>0$ and $K\in \mathbb N$.
Let $S$ be given by the Case 2 property.
Then there exists $k \geq K$ such that
\[ \|\chi_{B_R}[b,f_k](1-\chi_{B_S})\| \leq \frac\e5, \]
and in addition, $\|[b,f_{k}]\| > \e$ and $f_{k}|_{B_{S}}$ is $\frac\e5$-approximately constant.
From the latter property, it follows that there is a scalar $\gamma$ such that $f_{k}|_{B_{S}} \approx_{\e/10} \gamma$, so that
\begin{equation}
\label{eq:HigsonCommEquivalenceIneq1}
 \|\chi_{B_{S}}[b,f_k]\chi_{B_{S}}\| \leq 2\cdot\frac\e{10}\|b\| \leq \frac\e5.
\end{equation}
Since $b$ and $f_{k}$ are self-adjoint,
\begin{equation}
\label{eq:HigsonCommEquivalenceIneq2}
 \|(1-\chi_{B_{S}})[b,f_k]\chi_{B_{R}})\| = \|\chi_{B_{R}}[b,f_k](1-\chi_{B_{S}})\| \leq \frac\e5.
\end{equation}
We now cut up the operator $T=[b,f_k]$ as follows:

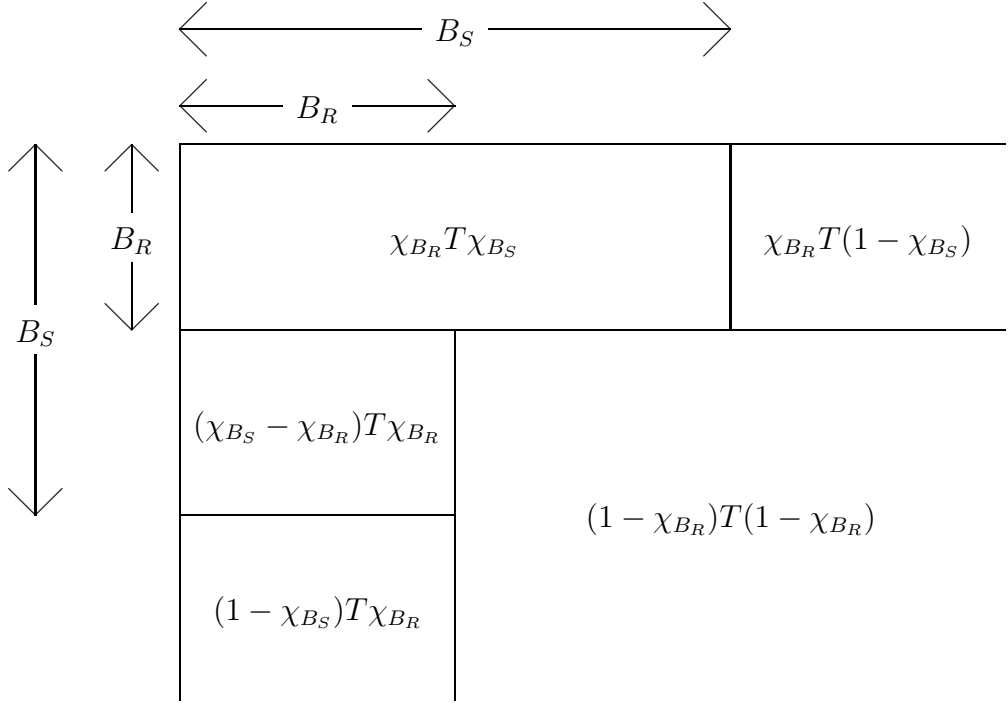
\setlength{\unitlength}{1pt}

\begin{picture}(381,264)
\put(150.5,170.5){$\chi_{B_R}T\chi_{B_S}$}  
\put(290.5,170.5){$\chi_{B_R}T(1-\chi_{B_S})$}  
\put(77,100.5){$(\chi_{B_S}-\chi_{B_R})T\chi_{B_R}$}  
\put(84.5,30.5){$(1-\chi_{B_S})T\chi_{B_R}$}  
\put(224,65.5){$(1-\chi_{B_R})T(1-\chi_{B_R})$}  
\put(72,210){\line(1,0){309}}  
\put(72,140){\line(1,0){309}}  
\put(72,70){\line(1,0){103}}  
\put(72,210){\line(0,-1){210}}  
\put(175,140){\line(0,-1){140}}  
\put(278,210){\line(0,-1){70}}  
\put(115.5,220){$B_R$}  
\put(167.5,249){$B_S$}  
\put(72,224.5){\line(1,0){38.5}}  
\put(175,224.5){\line(-1,0){38.5}}  
\put(72,224.5){\line(1,1){10}}  
\put(72,224.5){\line(1,-1){10}}  
\put(175,224.5){\line(-1,1){10}}  
\put(175,224.5){\line(-1,-1){10}}  
\put(72,253.5){\line(1,0){90.5}}  
\put(278,253.5){\line(-1,0){90.5}}  
\put(72,253.5){\line(1,1){10}}  
\put(72,253.5){\line(1,-1){10}}  
\put(278,253.5){\line(-1,1){10}}  
\put(278,253.5){\line(-1,-1){10}}  
\put(46,170.5){$B_R$}  
\put(10.5,135.5){$B_S$}  
\put(54,210){\line(0,-1){25.5}}  
\put(54,140){\line(0,1){25.5}}  
\put(54,210){\line(1,-1){10}}  
\put(54,210){\line(-1,-1){10}}  
\put(54,140){\line(1,1){10}}  
\put(54,140){\line(-1,1){10}}  
\put(18,210){\line(0,-1){60.5}}  
\put(18,70){\line(0,1){60.5}}  
\put(18,210){\line(1,-1){10}}  
\put(18,210){\line(-1,-1){10}}  
\put(18,70){\line(1,1){10}}  
\put(18,70){\line(-1,1){10}}

\end{picture}

That is, we use the equality
\begin{align*}
T &= (1-\chi_{B_R})T(1-\chi_{B_R}) +
\chi_{B_R}T\chi_{B_S} + \chi_{B_R}T(1-\chi_{B_S}) \\
&\qquad \quad + (\chi_{B_S}-\chi_{B_R})T\chi_{B_R} 
+ (1-\chi_{B_S})T\chi_{B_R} 
\end{align*}
and the reverse triangle inequality to deduce 
\begin{eqnarray}
\notag
&&\hspace*{-4em} \|(1-\chi_{B_{R}})[b,f_k](1-\chi_{B_{R}})\| \\
\notag
&\geq&
\|[b,f_k]\| - \big(
\|\chi_{B_R}[b,f_k]\chi_{B_S}\| + \|\chi_{B_R}[b,f_k](1-\chi_{B_S})\| \\
\notag
&&\qquad
+ \|(\chi_{B_S}-\chi_{B_R})[b,f_k]\chi_{B_R}\| + \|(1-\chi_{B_S})[b,f_k]\chi_{B_R}\|\big) \\
\notag
&\stackrel{\eqref{eq:HigsonCommEquivalenceIneq1},\eqref{eq:HigsonCommEquivalenceIneq2}}\geq&
\|[b,f_k]\| - 4\cdot\frac\e5 \\ 
\notag 
&>& \e-\frac{4\e}5  \\
&=& \frac\e5.
\end{eqnarray}
In summary, we have shown that for every $R>0$ and $K \in \mathbb N$, there exists $k \geq K$ such that 
\begin{equation}
\label{eq:HigsonCommEquivalenceRevTriangle}
\|(1-\chi_{B_{R}})[b,f_k](1-\chi_{B_{R}})\| > \frac\e5.
\end{equation}

Now, start with $R_0:=1$ and $k_0 := 0$, and (as in Case 1) we will choose $k_1 < k_2 < \cdots$ and $R_1 \geq 6R_0$, $R_2 \geq 6R_1$, $\dots$ recursively.
Given $R_{i-1}$, pick $k > k_{i-1}$ satisfying \eqref{eq:HigsonCommEquivalenceRevTriangle} for $R:=6R_{i-1}$, and set $k_i$ equal to this $k$.
That is, $k_i>k_{i-1}$ satisfies
\[
\|(1-\chi_{B_{6R_{i-1}}})[b,f_{k_i}](1-\chi_{B_{6R_{i-1}}})\| > \frac\e5. \]
Then, since $(\chi_{B_R})_{R=1}^\infty$ converges strongly to $1$, there exists $R_i \geq 6R_{i-1}$ such that
\[
\|\chi_{B_{R_i}}(1-\chi_{B_{6R_{i-1}}})[b,f_{k_i}](1-\chi_{B_{6R_{i-1}}})\chi_{B_{R_i}}\| > \frac\e5. \]
Note that $\chi_{B_{R_i}}(1-\chi_{B_{6R_{i-1}}}) = \chi_{B_{R_i}}-\chi_{B_{6R_{i-1}}}$.

Using these recursive choices, define $g:=g_{x_0,(f_{k_i})_{i=1}^\infty,(R_i)_{i=0}^\infty} \in C_{lh}(X) \subseteq C_h(X)$ by Lemma \ref{lem:Higsong}.
If $[b,g] \in \mathcal K(X,B)$, then $\|[b,g](1-\chi_{B_S})\|\to 0$ as $S \to \infty$.
However, given $S >0$, there exists $i$ such that $6R_{i-1} > S$.
Then
\begin{align*}
\left\|[b,g]\left(1-\chi_{B_S}\right)\right\|
&\geq \left\|(\chi_{B_{R_i}}-\chi_{B_{6R_{i-1}}})[b,g](\chi_{B_{R_i}}-\chi_{B_{6R_{i-1}}})\right\| \\
&= \left\|(\chi_{B_{R_i}}-\chi_{B_{6R_{i-1}}})[b,f_{k_i}](\chi_{B_{R_i}}-\chi_{B_{6R_{i-1}}})\right\| \\
&> \frac\e5,
\end{align*}
again a contradiction.
This concludes the proof.
\end{proof}
\bigskip

\section{More about \texorpdfstring{$\SV_\oomega\left(X\right)$}{VL_infty(X)}}
\label{sec:MoreAboutVL}

\subsection{To what extent does \texorpdfstring{$\SV_\oomega\left(X\right)$}{SVomega(X)} determine \texorpdfstring{$X$}{X}?}

If $X$ is a metric space, then we say that a set $E \subseteq X \times X$ is \emph{uniformly bounded} (also called a \emph{metric entourage}) if there exists $R>0$ such that
\[ E \subseteq \{(x,y) \in X \times X: d(x,y) \leq R\}. \]

\begin{defn}
  \label{def:coarse-maps}
  Let $X$ and $Y$ be metric spaces. We say that a (not necessarily continuous) function $\phi: X\to Y$ is:
  \begin{itemize}
    \item \emph{bornologous}, if for every $R\geq0$ there exists $S\geq 0$, such that for all $x,y\in X$, $d(x,y)\leq R$ implies $d(\phi(x),\phi(y))\leq S$;
    \item \emph{cobounded}, if $f^{-1}(y)$ is bounded for every $y\in Y$;
    \item \emph{coarse}, if it is both cobounded and bornologous;
    \item \emph{a coarse equivalence}, if it is bornologous, and there exists a bornologous $\psi: Y\to X$, such that both $\psi\circ\phi$ and $\phi\circ\psi$ are uniformly close to the identity maps, i.e., their graphs are uniformly bounded subsets of $X\times X$ and $Y\times Y$ respectively. Note that in this case both $\phi$ and $\psi$ are automatically coarse;
    \item \emph{locally Lipschitz}, if there exist $\delta>0$ and $T\geq0$, such that $d(x,y)\leq\delta$ implies $d(f(x),f(y))\leq Td(x,y)$, $x,y\in X$.
    \item \emph{a Lip-coarse equivalence}, if it is a coarse equivalence, it is locally Lipschitz, and in the definition of coarse equivalence, $\psi$ can be chosen to be locally Lipschitz as well.
  \end{itemize}
\end{defn}

Note that traditionally, coarse geometry does not concern itself with local behaviour. However, as our main tool in this piece are Lipschitz functions, we will insist that the maps involved are locally Lipschitz. On the other hand, in the key setting in which the metric spaces involved are uniformly discrete, this requirement is automatic, and thus can be ignored.

\begin{prop}
  Let $X$ be a proper metric space. $\SV_\oomega\left(X\right)$ is a Lip-coarse invariant for $X$.
  More precisely, if $X$ and $Y$ are proper metric spaces which are Lip-coarsely equivalent via a (locally Lipschitz) map $\phi:X \to Y$, then composition by $\phi$ induces a $^*$-isomorphism $\SV_\oomega\left(Y\right) \to \SV_\oomega\left(X\right)$.
\end{prop}

\begin{proof}
  Suppose that $X,Y$ are Lip-coarsely equivalent, so that there are locally Lipschitz coarse maps $\phi:X \to Y$ and $\psi:Y \to X$ such that the graphs of $\phi\circ\psi$ and $\psi\circ\phi$ are uniformly bounded. Denote by $\delta>0$ and $T\geq0$ the constants of local Lipschitzness of $\phi$. As locally Lipschitz maps are continuous, $\phi$ and $\psi$
induce maps $\phi_*:l^\infty\left(\mathbb N,C_b(Y)\right)\to l^\infty\left(\mathbb N,C_b(X)\right)$ and $\psi_*:l^\infty\left(\mathbb N,C_b(X)\right) \to l^\infty\left(\mathbb N,C_b(Y)\right)$.

First we show that $\phi_*\left(\SV\left(Y\right)\right) \subseteq \SV\left(X\right)$.
Surely, let $\left(f_n\right)_{n=1}^\infty \in \SV\left(Y\right)$ with $\left\|\left(f_n\right)_{n=1}^\infty\right\|\leq 1$.
Let $L>0$.
Since $\phi$ is a coarse map, there exists $S>0$ such that if $x,y \in X$ satisfy $d(x,y)\leq 2/L$ then $d(\phi(x),\phi(y))<S$.
Without loss of generality, we can assume that $S>\delta T$.
Since $(f_n)_{n=1}^\infty \in \SV(Y)$, there exists $n_0$ such that $f_n$ is $(L\delta/S)$-Lipschitz for all $n \geq n_0$.
For $n\geq n_0$, let us show that $f_n \circ \phi$ is $L$-Lipschitz.
Let $x,y \in X$, there are three cases:
\begin{itemize}
  \item If $d(x,y) > 2/L$ then $|f_n(x)-f_n(y)|\leq 2\|f_n\| < Ld(x,y)$.
  \item If $\delta\leq d(x,y)\leq 2/L$, then $d(\phi(x),\phi(y))<S$. Since
      $f_n$ is $(L\delta/S)$-Lipschitz, $$\|f_n\circ\phi(x)-f_n\circ\phi(y)\| \leq L\delta \leq
      Ld(x,y).$$
  \item If $d(x,y)<\delta$, then
    $$
    \|f_n(\phi(x))-f_n(\phi(y))\| \leq \frac{L\delta}{S}d(\phi(x),\phi(y))\leq \frac{L\delta T}{S}d(x,y)<Ld(x,y).
    $$
\end{itemize}
Thus, $\phi_*$ induces a map $\theta_\phi:\SV_\oomega\left(Y\right) \to \SV_\oomega\left(X\right)$; likewise, $\psi_*$ induces a map $\theta_\psi:\SV_\oomega\left(X\right) \to \SV_\oomega\left(Y\right)$.
Let us show that these maps are inverses.
By the symmetry of their definition, it suffices to show that $\theta_\psi\circ\theta_\phi = \id_{\SV_\oomega\left(Y\right)}$.

Let $\left(f_n\right)_{n=1}^\infty \in \SV_\oomega\left(Y\right)$.
Let $\Gamma\left(\phi\circ\psi\right) \subset Y \times Y$ denote the graph of $\phi\circ\psi$.
Since this is uniformly bounded, given $\e>0$, we may find $n_0$ such that, for $n \geq n_0$ and $\left(x,y\right) \in \Gamma\left(\phi\circ\psi\right)$,
\[ |f_n\left(x\right)-f_n\left(y\right)|\leq\e. \]
In other words, for $n \geq n_0$, $|f_n\left(x\right)-f_n\left(\phi\left(\psi\left(x\right)\right)\right)|\leq\e$, and thus,
\[ \|f_n-\psi_*\circ\phi_*\left(f_n\right)\|\leq \e. \]
Consequently, $\left\|\left(f_n\right)_{n=1}^\infty-\theta_\psi\circ\theta_\phi\left(\left(f_n\right)_{n=1}^\infty\right)\right\|_{\SV_\oomega\left(Y\right)}\leq \e$.
Since $\e$ is arbitrary, $\theta_\psi\circ\theta_\phi\left(\left(f_n\right)_{n=1}^\infty\right) = \left(f_n\right)_{n=1}^\infty$ in $\SV_\oomega(Y)$.
\end{proof}
\bigskip

\begin{lemma}
\label{lem:ControlledChar}
Let $X$ be a metric space, and let $E \subseteq X \times X$.
The following are equivalent:
\begin{enumerate}
\item $E$ is uniformly bounded.
\item For every $\left(f_n\right)_{n=1}^\infty \in \SV\left(X\right)$ and $\e > 0$, there exists $n_0$ such that, if $n \geq n_0$ and $\left(x,y\right) \in E$ then $|f_n\left(x\right)-f_n\left(y\right)|<\e$.
\item For every $\left(f_n\right)_{n=1}^\infty \in \SV\left(X\right)$, there exists $n_1$ such that, if $\left(x,y\right) \in E$ then $|f_{n_1}\left(x\right)-f_{n_1}\left(y\right)|<1$.
\end{enumerate}
\end{lemma}

\begin{proof}
(i) $\Rightarrow$ (ii) is by the definition of $\SV\left(X\right)$.
(ii) $\Rightarrow$ (iii) is trivial.

For (iii) $\Rightarrow$ (i), let $E \subseteq X \times X$ be a set that isn't uniformly bounded, and let us show that there exists $\left(f_n\right)_{n=1}^\infty \in \SV\left(X\right)$ such that, for all $n$ there exists $\left(x_n,y_n\right)_{n=1}^\infty \in E$ such that $|f_n\left(x\right)-f_n\left(y\right)|=1$.

For each $n$, there exists $\left(x_n,y_n\right) \in E$ such that $d(x,y)>n$.
Thus there exists a $(1/n)$-Lipschitz function $f_n:X \to \R$ such that $f_n(x_n)=0$, and $f_n(y_n)=1$.
It follows that $\left(f_n\right)_{n=1}^\infty \in \SV\left(X\right)$, showing that (iii) doesn't hold.
\end{proof}
\bigskip

\begin{prop}
Let $X,Y$ be metric spaces.
Let $\phi:Y \to X$ be a function.
\begin{enumerate}
\item If $\phi_*\left(\SV\left(X\right)\right) \subseteq \SV\left(Y\right)$ then $\phi$ is a bornologous map;
\item If, moreover, the induced map $\theta_\phi:\SV_\oomega\left(X\right) \to \SV_\oomega\left(Y\right)$ is an isomorphism, then $\phi$ is a coarse equivalence.
\end{enumerate}
\end{prop}

\begin{proof}
(i):
Let $R>0$.
We must show that $E:=\{(\phi(x),\phi(y)):x,y\in X,\ d(x,y)<R\}$ is uniformly bounded.
To this end, we will verify Lemma \ref{lem:ControlledChar} (iii) for this set.
Therefore, let $\left(f_n\right)_{n=1}^\infty \in \SV\left(X\right)$.

Since $\phi_*\left(\left(f_n\right)_{n=1}^\infty\right) \in \SV\left(Y\right)$, there exists $n_1$ such that $f_{n_1} \circ \phi$ is $R^{-1}$-Lipschitz.
Thus for $(\phi(x),\phi(y)) \in E$, i.e., $d(x,y)<R$,
\[ |f_{n_1}(\phi(x))-f_{n_1}(\phi(x))| = |f_{n_1} \circ \phi(x) - f_{n_1} \circ \phi(y)| < 1, \]
as required.
\medskip

(ii):
We must show two things: (a) for every $R>0$, the set $\{(x,y) \in Y\times Y : d(\phi(x),\phi(y))<R\}$ is uniformly bounded, and (b) there exists $R>0$ such that, for all $x \in X$, there exists $y\in Y$ such that $d\left(x,\phi\left(y\right)\right) <R$.

(a):
Let $R\geq0$ be given.
We will verify Lemma \ref{lem:ControlledChar} (iii) for $E:=\{(x,y) \in Y\times Y : d(\phi(x),\phi(y))<R\}$.
Therefore, let $\left(f_n\right)_{n=1}^\infty \in \SV\left(Y\right)$.

Since $\theta_\phi$ is surjective, there exists $\left(g_n\right)_{n=1}^\infty \in \SV\left(X\right)$ such that $\lim_\oomega \|f_n-g_n\circ\phi\| = 0$.
By Lemma \ref{lem:ControlledChar} (i) $\Rightarrow$ (ii), there exists $n_0$ such that, if $n \geq n_0$ and $w,z \in X$ satisfy $d(w,z) < R$ then $g_n(w) \approx_{1/3} g_n(z)$.
Pick $n \geq n_0$ such that $\|f_n-g_n\circ\phi\| < \frac13$.

Now, let $\left(x,y\right) \in E$, i.e., $d(\phi(x),\phi(y))<R$.
Then
\[ f_n\left(x\right) \approx_{1/3} g_n\left(\phi\left(x\right)\right) \approx_{1/3} g_n\left(\phi\left(y\right)\right) \approx_{1/3} f_n\left(y\right), \]
as required.

(b):
Proof by contradiction.
Suppose for a contradiction that, for every $n$ there exists $x_n \in X$ such that for all $y \in Y$,
\[ \left(x_n,\phi\left(y\right)\right) \geq n. \]

Thus, there exists a $(1/n)$-Lipschitz function $f_n:X\to [0,1]$ such that $f_n(x_n)=1$ and $f_n(\phi(y))=0$ for all $y \in Y$.
Putting these together, we obtain $\left(f_n\right)_{n=1}^\infty \in \SV\left(X\right)$ and $\|f_n\|=1$ (since $f_n\left(x_n\right)=1$), so that $\left\|\left(f_n\right)_{n=1}^\infty\right\|=1$ in $\SV_\oomega\left(X\right)$.
However, since $f_n\left(\phi\left(y\right)\right) = 0$ for all $y \in Y$, it follows that $\theta_\phi\left(\left(f_n\right)_{n=1}^\infty\right) = 0$, which contradicts injectivity of $\theta_\phi$.
\end{proof}
\bigskip

In other words, when $\SV_\oomega\left(X\right) \cong \SV_\oomega\left(Y\right)$, and the isomorphism comes from a map between $Y$ and $X$, it follows that $X$ and $Y$ are coarsely equivalent.
Here is the more interesting question:

\begin{question}
Let $X,Y$ be uniformly discrete metric spaces.
If $\SV_\oomega\left(X\right) \cong \SV_\oomega\left(Y\right)$, must $X$ and $Y$ be coarsely equivalent?
\end{question}
\bigskip

\subsection{The nuclear dimension of \texorpdfstring{$\SV_\oomega\left(X\right)$}{SVomega(X)}}

$\SV\left(X\right)$ and $\SV_\oomega\left(X\right)$ are commutative unital C*-algebras, and therefore by Gelfand's Theorem, each are algebras of continuous functions on a compact Hausdorff space, namely the Gelfand spectrum of the respective algebras.
As these C*-algebras are nonseparable, their spectra are nonmetrizable.
Here we show a relationship between the asymptotic dimension of $X$ and the covering dimension (suitably interpreted) of these spectra.
In fact we use the nuclear dimension of the algebras, which (for the spectra) corresponds to a version of covering dimension which is slightly modified (in this nonseparable case) from the original definition.

\begin{defn}\cite[Definition 2.1]{WinterZacharias:NucDim}
Let $A$ be a C*-algebra and let $d \in \N$.
We say that the \emph{nuclear dimension} of $A$ is at most $d$ if there exists a net $(F_\lambda,\psi_\lambda,\phi_\lambda)$ where $F_\lambda$ is a finite dimensional C*-algebra, $\psi_\lambda:A \to F_\lambda$ and $\phi_\lambda:F_\lambda \to A$ are completely positive maps such that
\[ \lim_\lambda \psi_\lambda(\phi_\lambda(a)) = a, \quad a \in A, \]
$\psi_\lambda$ is contractive, and $F_\lambda$ decomposes into direct summands as $F_\lambda = F_\lambda^{(0)}\oplus \cdots \oplus F_\lambda^{(d)}$ such that $\phi_\lambda|_{F_\lambda^{(i)}}$ is contractive and order zero, for each $i$.
\end{defn}

Since we will be considering nuclear dimension for the commutative and nonseparable C*-algebras $\SV(X)$ and $\SV_\oomega(X)$, let us explain exactly the modification to covering dimension that is entailed by nuclear dimension.
Let $Y$ be a locally compact Hausdorff space; call an open set $U\subseteq Y$ a \emph{preimage-open} set if it is the preimage of an open subset of $\mathbb R$ under a continuous function $Y \to \mathbb R$.
Then the nuclear dimension of $C_0(Y)$ is equal to the smallest number $d$ such that every finite cover of $Y$ consisting of preimage-open sets has a $(d+1)$-colourable refinement consisting of preimage-open sets (see the proof of \cite[Proposition 3.3]{KirchbergWinter:CovDim}, and \cite[Proposition 2.4]{WinterZacharias:NucDim}; this fact is alluded to in the discussion before \cite[Proposition 2.4]{WinterZacharias:NucDim}).
(In the second countable situation, or more generally when $Y$ is a normal space, all open sets are preimage-open, which is why nuclear dimension coincides with the usual definition of covering dimension in this case.)

\begin{defn} (\cite[\S 1.E]{Gromov:AsympInvariants})
\label{def:AsympDim}
Let $X$ be a metric space. Then the asymptotic dimension of $X$ is at most $d\in \mathbb N$, written $\asdim(X) \leq d$, if for every $R>0$, there exists a cover of $X$ of the form $(U^{(i)}_j)_{i=0,\dots,d;\,j\in J}$, such that for each $i=0,\dots,d$, the family $(U^{(i)}_j)_{j\in J}$ is $R$-disjoint and uniformly bounded.
\end{defn}

\begin{prop}
Let $X$ be a metric space.
\label{prop:dimNucEasyDirection}
$\dn\,\SV\left(X\right) \leq \asdim(X)$ and 
$\dn\, \SV_\oomega\left(X\right) \leq \asdim(X)$.
\end{prop}

\begin{proof}
As the nuclear dimension decreases when passing to quotients, it suffices to prove the first statement.

Set $d:=\asdim(X)$.
Let $\mathcal F \subset \SV(X)$ be a finite set and let $\e>0$ be given.

Using the definition of asymptotic dimension in a fairly straightforward way, for each $n \in \N$, we may find an infinite partition of unity $(e_j^{(i)}(n))_{j\in J(n);\, i=0,\dots,d}$, such that:
\begin{enumerate}
\item for each $i$, $(e_j^{(i)}(n))_{j\in J(n)}$ are pairwise orthogonal,
\item each $e_j^{(i)}$ is $(1/n)$-Lipschitz, and
\item there is a uniform bound, $S(n)$, on the diameters of the supports of $e_j^{(i)}(n)$ (allowed to depend only on $n$).
\end{enumerate}
Let us also pick a point $x_j^{(i)}(n)$ inside the support of $e_j^{(i)}(n)$, for each $i,j,$ and $n$.
For $n=0$, define $J(0):=X$, $e_j^{(0)}(0)(x):=\delta_{j,x}$, $e_j^{(i)}(0):=0$ for $i>0$, $S(0):=0$, and $x_j^{(i)}(0)=j \in X$.

Define a $^*$-homomorphism $\psi_n = (\psi_n^{(0)},\dots,\psi_n^{(d)}):C_b(X) \to \bigoplus_{i=0}^d l^\infty(J(n))$ coordinatewise by evaluation at $x_j^{(i)}(n)$.
For $i=0,\dots,n$, define a $^*$-homomorphism $\phi_n^{(i)}:l^\infty(J(n)) \to C_b(X)$ by
\[ \phi_n^{(i)}(\lambda) = \sum_{j \in J(n)} \lambda(j)e_j^{(i)}(n), \]
with the sum converging pointwise, since in fact at each point $x \in X$, at most one summand is nonzero (by condition (i)).
Note that if $f \in C_b(X)$ is $(\e/S(n))$-Lipschitz then
\begin{equation}
\label{eq:DimNucApprox}
 \sum_{i=0}^d \phi_n^{(i)}\circ\psi_n^{(i)}(f) \approx_\e f.
\end{equation}

Let
\[ \mathcal F = \{(f_{i,k})_{k=1}^\infty : i=1,\dots,m\}. \]
For each $k$, we may find some $n_k \geq 0$ such that $f_{i,k}$ has $(\e/S(n_k))$-Lipschitz for all $i=1,\dots,m$.
Since $\mathcal F \subseteq \SV(X)$, we can pick these $n_k$ such that they converge to $\infty$.

We now define a $^*$-homomorphism
\[ \Psi := (\psi_{n_k})_{k=1}^\infty:l^\infty(\N,l^\infty(X)) \to \prod_k \bigoplus_{i=0}^d l^\infty(J(n_k)) \cong \bigoplus_{i=0}^d \prod_k l^\infty(J(n_k)); \]
we may write $\Psi=(\Psi^{(0)},\dots,\Psi^{(d)})$.
For $i=0,\dots,d$, define a $^*$-homomorphism
\[ \Phi^{(i)} := (\phi^{(i)}_{n_k})_{k=1}^\infty:\prod_k l^\infty(J(n_k)) \to l^\infty(\N,l^\infty(X)). \]
Since $n_k \to \infty$, we see that the image of $\Phi^{(i)}$ is in fact contained in $\SV(X)$.
Moreover, by \eqref{eq:DimNucApprox} and our choice of $n_k$, we find that
\[ \sum_{i=0}^d \Phi^{(i)}\circ\Psi^{(i)}(f) \approx_\e f \]
for $f \in \mathcal F$.
Since $\prod_k l^\infty(J(n_k))$ has nuclear dimension zero, this is sufficient to prove that $\SV(X)$ has nuclear dimension at most $d$.
\end{proof}
\bigskip

We have an argument to get inequalities in the other direction, under the hypothesis that $X$ has finite asymptotic dimension.
For this, we begin with the following lemma.

\begin{lemma}
\label{lem:dimNucTechnical}
Let $X$ be a set and let $\eta>0$.
Let $f_1,\dots,f_m, e_1,\dots,e_n:X \to [0,\infty)$, and $\lambda_{i,j} \in [0,\infty)$ for $i=1,\dots,n,\ j=1,\dots,m$.
Suppose that $f_j \approx_\eta \sum_{i=1}^n \lambda_{i,j} e_i$ for each $j$ and $\sum_{j=1}^m \lambda_{i,j} = 1$ for each $i$.
Then for each $i$, there exists $j(i)$ such that
\[ \{x: e_i(x) > m\eta\} \subseteq \{x : f_{j(i)}(x) > 0\}. \]
\end{lemma}

\begin{proof}
Fix $i$.
Since $\sum_{j=1}^m \lambda_{i,j} =1$, there exists some $j=j(i)$ such that $\lambda_{i,j} \geq 1/m$.
For $x \in X$ such that $f_{j}(x)=0$, it follows that
\begin{align*}
\frac1m e_i(x) &\leq \lambda_{i,j} e_i(x) \\
&\leq \sum_{i'=1}^n \lambda_{i',j} e_{i'}(x) \\
&\leq f_j(x)+\eta = \eta.
\end{align*}
This shows that $e_i(x) \leq m\eta$, as required.
\end{proof}
\bigskip

\begin{thm}
\label{thm:NucDim}
If $X$ has finite asymptotic dimension, then
\[ \asdim(X) = \dn\, \SV(X) = \dn\, \SV_\oomega(X). \]
\end{thm}

\begin{proof}
Set $d:=\dn\, \SV_\oomega(X)$.
By Proposition \ref{prop:dimNucEasyDirection}, it suffices to show that $\asdim(X) \leq d$.
Let $R>0$ be given, and we will partition $X$ into $(d+1)$ uniformly bounded, $R$-disjoint families.

By hypothesis, let $\asdim(X) \leq m-1$.
Then from this, there exists a partition of unity $g_1,\dots,g_m \in \SV_\oomega(X)$, such that $g_j=(g_{j,l})_{l=1}^\infty$ where for each $j,l$, the support of $g_{j,l}$ decomposes as an $l$-disjoint, uniformly bounded family of subsets of $X$. 

Set
\[ \eta := \frac1{3(d+1)m}. \]
The only nonzero order zero maps from a matrix algebra into a commutative algebra occur when the matrix algebra is one-dimensional (this follows from \cite[Proposition 3.2 (a)]{Winter:CovDim1}).
Hence, $\dn\, \SV(X) \leq d$ implies that there exists $s \in \mathbb N$, a c.p.c.\ map $\psi=(\psi^{(0)},\dots,\psi^{(d)}):\SV_\oomega(X) \to \bigoplus_{i=0}^d \mathbb C^{\oplus s}$, and c.p.c.\ order zero maps $\phi^{(i)}:\mathbb C^s \to \SV_\oomega(X)$ such that
\[  g_j \approx_\eta \sum_{i=0}^m \phi^{(i)}\circ\psi^{(i)}(g_j) \quad \text{and} \quad 1_{\SV_\oomega(X)} \approx_{1/2} \sum_{i=0}^m \phi^{(i)}\circ\psi^{(i)}(1). \]
By rescaling, we may assume that $\psi^{(i)}(1_{\SV_\oomega(X)})=(1,\dots,1)$.
Write $\psi^{(i)}(g_j) = (\lambda_{(i,1),j},\dots,\lambda_{(i,s),j}) \in [0,\infty)^s$ (since $\psi$ is positive).
By linearity, looking at the $i'$ component of $\psi^{(i)}(1)$, we have
\[ \sum_{j=1}^m \lambda_{(i,i'),j} = 1. \]

Each map $\phi^{(i)}$ lifts to a c.p.c.\ order zero map $(\phi^{(i)}_l)_{l=1}^\infty:\mathbb C^s \to \SV(X)$, by \cite[Remark 2.4]{KirchbergWinter:CovDim}.
For all but finitely many $l$, we have
\begin{align}
\notag
g_{j,l} &\approx_\eta \sum_{i=0}^d \phi^{(i)}_l(\lambda_{(i,1),j},\dots,\lambda_{(i,s),j}) \quad &&\text{and} \\
1_{C_b(X)} &\approx_{1/2} \sum_{i=0}^d \phi^{(i)}_l(1,\dots,1). 
\label{eq:dimNucCrucial1}
\end{align}

Fix $l \geq R$ for which \eqref{eq:dimNucCrucial1} holds, and such that the image of each minimal projection in $\mathbb C^s$ under $\phi^{(i)}_l$ is $(m\eta/L)$-Lipschitz.
Write $e_{(i,1)},\dots,e_{(i,s)}$ for these images under $\phi^{(i)}_l$ of the minimal projections in $\mathbb C^s$, and write $f_j:=g_{j,l}$; thus \eqref{eq:dimNucCrucial1} becomes
\begin{equation}
\label{eq:dimNucCrucial2}
 f_j \approx_\eta \sum_{i=0}^d \sum_{i'=1}^s \lambda_{(i,i'),j}e_{(i,i')} \quad \text{and}\quad 1 \approx_{1/2} \sum_{i=0}^d \sum_{i'=1}^s e_{(i,i')}. 
\end{equation}

We now apply Lemma \ref{lem:dimNucTechnical} with $(i,i'),\ i=0,\dots,d,\ i'=1,\dots,s$ in place of the index $i=1,\dots,n$.
This tells us that for each $i=0,\dots,d$ and $i'=1,\dots,s$, there exists some $j(i,i')$ such that
\[ B_{i,i'}:=\{x \in X: e_{(i,i')} > m\eta\} \subseteq \{x \in X : f_{j(i,i')}(x) >0\}. \]
Since the support of $f_{j(i,i')}$ ($= g_{j(i,i'),l}$) decomposes as a union of an $l$-disjoint uniformly bounded family, and $l \geq R$, we can partition $B_{i,i'}$ into an $R$-disjoint, uniformly bounded family, say
\[ B_{i,i'} = \coprod_{t\in T} A^{(i)}_{i',t} \]
Fixing $i \in \{0,\dots,d\}$, we now consider the family $(A^{(i)}_{i',t})_{i'=1,\dots,s,\ t \in T}$.
This family is a finite union of uniformly bounded families, whence it is uniformly bounded.
Let us check that it is $R$-disjoint.
Since for fixed $i'$ we already have $R$-disjointness of $(A^{(i)}_{i',t})_{t \in T}$, we need to show that for $i'_1 \neq i'_2$ and $t_1,t_2 \in T$, the minimal distance between $A^{(i)}_{i'_1,t_1}$ and $A^{(i)}_{i'_2,t_2}$ is at least $R$.
In other words, we need to show that the minimal distance between $B_{i,i'_1}$ and $B_{i,i'_2}$ is at least $R$.

We have that $e_{i,i'_1}$ and $e_{i,i'_2}$ are orthogonal, so if $x \in B_{i,i'_1}$ then since $e_{i,i'_1}(x) \neq 0$, it must be the case that $e_{i,i'_2}(x)=0$.
Consider now $y \in B_{i,i'_2}$, so that $e_{i,i'_2}(y) > m\eta$.
Since $e_{i,i'_2}$ is $(m\eta/R)$-Lipschitz, it follows that $d(x,y) \geq R$, as required.

Finally let us show that 
\[ (A^{(i)}_{i',t})_{i=0,\dots,d,\ i'=1,\dots,s,\ t\in T} \]
is a cover of $X$, i.e., that $X=\bigcup_{i,i'} B_{i,i'}$.
For $x \in X$, from the second part of \eqref{eq:dimNucCrucial2}, we have
\[ 1/2 <  \sum_{i=0}^d \sum_{i'=1}^s e_{(i,i')}(x). \]
At most $d+1$ terms in this sum are nonzero, due to the pairwise orthogonality withing each family $(e_{(i,i')})_{i'=1}^s$.
Therefore, there exists some $i,i'$ such that $e_{(i,i')}(x) > \frac1{2(d+1)} = m\eta$.
Thus, $x \in B_{i,i'}$ (by definition) as required.
\end{proof}
\bigskip

We have the following consequence.

\begin{cor}
\label{cor:dimNuc}
Suppose $X$ is a metric space, and for all $m\in \mathbb N$, $X$ contains a subspace $Y_m$ such that $\asdim(Y_m) \in [m,\infty)$.
Then 
\[ \dn\, \SV(X) = \dn\, \SV_\oomega(X) = \infty. \]
\end{cor}

\begin{proof}
It is not too hard to see that restriction to $Y_m$ produces a surjective $^*$-homomorphism $\SV_\oomega(X) \to \SV_\oomega(Y_m)$ (for surjectivity, the key point is that an $L$-Lipschitz function on a closed subset of $Y_m$ extends to an $L$-Lipschitz function on $X$).
Hence we have 
\[ \dn\, \SV_\oomega(X) \geq \dn\, \SV_\oomega(Y_m) \geq m, \]
using \cite[Proposition 2.3(iv)]{WinterZacharias:NucDim} for the first inequality and Theorem \ref{thm:NucDim} for the second.
\end{proof}
\bigskip

In \cite[Theorem 7.2]{Dranishnikov:AsympTop}, it was shown that the asymptotic dimension of $X$ is equal to the covering dimension of the Higson corona $\nu X$, \emph{likewise provided that $\asdim(X) < \infty$}.

\begin{question}
Is $\dn \SV(X) = \asdim(X)$ always?
Is $\dn \SV(X) = \dim(\nu X)$ always?
\end{question}

In light of Corollary \ref{cor:dimNuc} and \cite[Theorem 7.2]{Dranishnikov:AsympTop}, the above question is only open in the case of a metric space $X$ of infinite dimension, which does not contain subspaces of arbitrarily large finite dimension.

\newcommand{\cstar}{C*}


\begin{thebibliography}{10}

\bibitem{BellDranishnikov}
G.~Bell and A.~Dranishnikov.
\newblock Asymptotic dimension.
\newblock {\em Topology Appl.}, 155(12):1265--1296, 2008.

\bibitem{Dranishnikov:AsympTop}
A.~N. Dranishnikov.
\newblock Asymptotic topology.
\newblock {\em Uspekhi Mat. Nauk}, 55(6(336)):71--116, 2000.

\bibitem{Dranishnikov:ScalarCurvatureConjecture}
Alexander Dranishnikov.
\newblock On {G}romov's positive scalar curvature conjecture for duality
  groups.
\newblock {\em J. Topol. Anal.}, 6(3):397--419, 2014.

\bibitem{DranishnikovZarichnyi:HaversC}
Alexander Dranishnikov and Michael Zarichnyi.
\newblock Asymptotic dimension, decomposition complexity, and {H}aver's
  property {C}.
\newblock {\em Topology Appl.}, 169:99--107, 2014.

\bibitem{Engel:PseudoDO}
Alexander Engel.
\newblock Index theory of uniform pseudodifferential operators.
\newblock arXiv preprint 1502.00494v1.

\bibitem{Engel:Rough}
Alexander Engel.
\newblock Rough index theory on spaces of polynomial growth and
  contractibility.
\newblock arXiv preprint 1505.03988v3.

\bibitem{Georgescu:JFA}
Vladimir Georgescu.
\newblock On the structure of the essential spectrum of elliptic operators on
  metric spaces.
\newblock {\em J. Funct. Anal.}, 260(6):1734--1765, 2011.

\bibitem{Georgescu:arXiv}
Vladimir Georgescu.
\newblock On the essential spectrum of the operators in certain crossed
  products, 2017.
\newblock arXiv preprint 1705.00379.

\bibitem{GeorgescuIftimovici}
Vladimir Georgescu and Andrei Iftimovici.
\newblock Localizations at infinity and essential spectrum of quantum
  {H}amiltonians. {I}. {G}eneral theory.
\newblock {\em Rev. Math. Phys.}, 18(4):417--483, 2006.

\bibitem{Gromov:AsympInvariants}
M.~Gromov.
\newblock Asymptotic invariants of infinite groups.
\newblock In {\em Geometric group theory, {V}ol.\ 2 ({S}ussex, 1991)}, volume
  182 of {\em London Math. Soc. Lecture Note Ser.}, pages 1--295. Cambridge
  Univ. Press, Cambridge, 1993.

\bibitem{Gromov:PolyGrowth}
Mikhael Gromov.
\newblock Groups of polynomial growth and expanding maps.
\newblock {\em Inst. Hautes \'Etudes Sci. Publ. Math.}, (53):53--73, 1981.

\bibitem{GuentnerKaminker}
Erik Guentner and Jerome Kaminker.
\newblock Exactness and the {N}ovikov conjecture.
\newblock {\em Topology}, 41(2):411--418, 2002.

\bibitem{GTY:Inv}
Erik Guentner, Romain Tessera, and Guoliang Yu.
\newblock A notion of geometric complexity and its application to topological
  rigidity.
\newblock {\em Invent. Math.}, 189(2):315--357, 2012.

\bibitem{GTY:GGD}
Erik Guentner, Romain Tessera, and Guoliang Yu.
\newblock Discrete groups with finite decomposition complexity.
\newblock {\em Groups Geom. Dyn.}, 7(2):377--402, 2013.

\bibitem{KasparovYu}
Gennadi Kasparov and Guoliang Yu.
\newblock The coarse geometric {N}ovikov conjecture and uniform convexity.
\newblock {\em Adv. Math.}, 206(1):1--56, 2006.

\bibitem{KirchbergWinter:CovDim}
Eberhard Kirchberg and Wilhelm Winter.
\newblock Covering dimension and quasidiagonality.
\newblock {\em Internat. J. Math.}, 15(1):63--85, 2004.

\bibitem{LangeRabinovich}
B.~V. Lange and V.~S. Rabinovich.
\newblock Noethericity of multidimensional discrete convolution operators.
\newblock {\em Mat. Zametki}, 37(3):407--421, 462, 1985.

\bibitem{Ozawa:AmenableActions}
Narutaka Ozawa.
\newblock Amenable actions and exactness for discrete groups.
\newblock {\em C. R. Acad. Sci. Paris S\'er. I Math.}, 330(8):691--695, 2000.

\bibitem{RRS}
V.~S. Rabinovich, S.~Roch, and B.~Silbermann.
\newblock Fredholm theory and finite section method for band-dominated
  operators.
\newblock {\em Integral Equations Operator Theory}, 30(4):452--495, 1998.
\newblock Dedicated to the memory of Mark Grigorievich Krein (1907--1989).

\bibitem{Roe:Email}
John Roe.
\newblock Personal communication.

\bibitem{Roe:JDG1}
John Roe.
\newblock An index theorem on open manifolds. {I}.
\newblock {\em J. Differential Geom.}, 27(1):87--113, 1988.

\bibitem{Roe:CBMS}
John Roe.
\newblock {\em Index theory, coarse geometry, and topology of manifolds},
  volume~90 of {\em CBMS Regional Conference Series in Mathematics}.
\newblock Published for the Conference Board of the Mathematical Sciences,
  Washington, DC; by the American Mathematical Society, Providence, RI, 1996.

\bibitem{Roe:LectOnCoarseGeom}
John Roe.
\newblock {\em Lectures on coarse geometry}, volume~31 of {\em University
  Lecture Series}.
\newblock American Mathematical Society, Providence, RI, 2003.

\bibitem{Schick:TopologyOfPSC}
Thomas Schick.
\newblock The topology of positive scalar curvature.
\newblock In {\em Proceedings of the International Congress of Mathematicians
  Seoul 2014}, volume~II, pages 1285--1308. Kyung Moon SA Co. Ltd., 2014.

\bibitem{SkandalisTuYu}
G.~Skandalis, J.~L. Tu, and G.~Yu.
\newblock The coarse {B}aum-{C}onnes conjecture and groupoids.
\newblock {\em Topology}, 41(4):807--834, 2002.

\bibitem{SpakulaWillett:Rigidity}
J\'an \v{S}pakula and Rufus Willett.
\newblock On rigidity of {R}oe algebras.
\newblock {\em Adv. Math.}, 249:289--310, 2013.

\bibitem{Winter:CovDim1}
Wilhelm Winter.
\newblock Covering dimension for nuclear {\cstar}-algebras.
\newblock {\em J. Funct. Anal.}, 199(2):535--556, 2003.

\bibitem{WinterZacharias:NucDim}
Wilhelm Winter and Joachim Zacharias.
\newblock The nuclear dimension of {\cstar}-algebras.
\newblock {\em Adv. Math.}, 224(2):461--498, 2010.

\bibitem{Wright:C0CoarseGeometry}
Nick Wright.
\newblock {$C_0$} coarse geometry and scalar curvature.
\newblock {\em J. Funct. Anal.}, 197(2):469--488, 2003.

\bibitem{Yu:ZeroInTheSpectrumConjecture}
Guoliang Yu.
\newblock Zero-in-the-spectrum conjecture, positive scalar curvature and
  asymptotic dimension.
\newblock {\em Invent. Math.}, 127(1):99--126, 1997.

\bibitem{Yu:NCfad}
Guoliang Yu.
\newblock The {N}ovikov conjecture for groups with finite asymptotic dimension.
\newblock {\em Ann. of Math. (2)}, 147(2):325--355, 1998.

\bibitem{Yu:UnifEmbeddableBCC}
Guoliang Yu.
\newblock The coarse {B}aum-{C}onnes conjecture for spaces which admit a
  uniform embedding into {H}ilbert space.
\newblock {\em Invent. Math.}, 139(1):201--240, 2000.

\end{thebibliography}
\end{document}